\documentclass[11pt]{article}

\usepackage{amsmath}
\usepackage{amssymb}
\usepackage{amsfonts}
\usepackage{wasysym}
\usepackage[unicode, pdftex, final]{hyperref}
\hypersetup{
	colorlinks = true,
	linkcolor=magenta,
	citecolor = red
}
\usepackage{amsthm}

\newtheorem{theorem}{Theorem}[section]
\newtheorem{proposition}[theorem]{Proposition}
\newtheorem{lemma}[theorem]{Lemma}
\newtheorem{corollary}[theorem]{Corollary}
\theoremstyle{remark}
\newtheorem*{remark}{Remark}
\newtheorem*{notation_remark}{Notation remark}
\theoremstyle{definition}
\newtheorem{definition}{Definition}

\usepackage{mathtools}
\numberwithin{equation}{section}
\newcommand{\normH}[2]{\norm{#2}_{\dot{H}_{#1}}}
\newcommand{\Gam}[2]{\Gamma\biggl[\genfrac{}{}{0pt}{}{#1}{#2}\biggr]}
\newcommand{\FO}[3]{{}_1F_1\biggl[\genfrac{}{}{0pt}{}{#1}{#2}\biggm|#3\biggr]}

\newcommand{\R}{\mathbb{R}}
\newcommand{\N}{\mathbb{N}}

\renewcommand{\Re}{\mathfrak{Re}}
\renewcommand{\Im}{\mathfrak{Im}}
\newcommand{\I}{\mathbb{I}}
\newcommand{\F}{\mathcal{F}}
\renewcommand{\P}{\mathbb{P}}
\newcommand{\E}{\mathbb{E}}
\newcommand{\CT}{\mathcal{T}_s}
\newcommand{\CO}[1]{G_{#1}}

\DeclarePairedDelimiter\abs{\lvert}{\rvert}
\DeclarePairedDelimiter\norm{\lVert}{\rVert}
\DeclarePairedDelimiter\bra{(}{)}

\DeclareMathOperator{\sign}{sgn}
\DeclareMathOperator{\Tr}{Tr}

\DeclareMathOperator{\supp}{supp}
\DeclareMathOperator{\Conf}{Conf}

\title{Central limit theorem for the determinantal point process with the confluent hypergeometric kernel}
\author{Sergei M. Gorbunov\footnote{Steklov Mathematical Institute of Russian Academy of Sciences, Moscow, Russia\\
This work was supported by the Russian Science Foundation under grant no. 24-71-10109, https://rscf.ru/en/project/24-71-10109/.}}
\date{}

\begin{document}
\maketitle
\begin{abstract}
We consider the convergence of additive functionals under the determinantal point pro-
cess with the confluent hypergeometric kernel, corresponding to a sufficiently smooth function
$f(x/R)$, as $R\to\infty$. We show that these functionals approach Gaussian distribution and give an
estimate on the Kolmogorov-Smirnov distance. To obtain these results we derive an exact identity for expectations of multiplicative functionals in terms of Fredholm determinants.
\end{abstract}
\section{Introduction}
Fix a complex number $s$ such that $\Re s >-1/2$. For $x\ne y\in\R$ consider the following kernel
\begin{equation}\label{1_eq:CHK_def}
        K^s(x, y) = \rho(x)\rho(y)\frac{Z_s(x)\overline{Z_s(y)}-e^{i(x-y)}\overline{Z_s(x)}Z_s(y)}{2\pi i(y-x)},
        \end{equation}
        where
            \begin{align*}
    &\Gam{a, b,\ldots}{c, d\ldots} = \frac{\Gamma(a)\Gamma(b)\ldots}{\Gamma(c)\Gamma(d)\ldots},\quad \rho(x)=\abs{x}^{\Re s}e^{-\frac{\pi}{2}\Im s\sign x},\\
    &Z_s(x) = \Gam{1+s}{1+2\Re s}\FO{\bar{s}}{1+2\Re s}{ix},
    \end{align*}
    and ${}_1F_1$ stands for the confluent hypergeometric function \cite[Sect.~13]{AS_64}, defined by the formula
    \[
    \FO{a}{b}{z} = \sum_{k=0}^\infty \frac{(a)_k}{(b)_kk!}z^k,\qquad (a)_k = a(a+1)\ldots (a+k-1).
    \]
    
    For $x=y$ define $K^s(x, x)$ by the L'H\^opital rule. The kernel induces a locally trace class operator of orthogonal projection on $L_2(\R)$ (see Theorem \ref{3:Diagonalization} or \cite[Corollary~1]{B_23}) and by the Macchi-Soshnikov Theorem \cite{M_75, S_00} induces a determinantal point process $\P^s$.
    
    Recall that $\P^s$ is a measure on discrete subsets of $\R$ without accumulation points, which we denote by $\Conf(\R)$. For a compactly supported Borel bounded function define an additive functional
    \[
    S_f(X) = \sum_{x\in X}f(x), \quad X\in\Conf(\R).
    \]
    Under the measure $\P^s$ it becomes a random variable, which has all moments for $f\in L_\infty(\R)\cap L_1(\R)$. One may regularize it by subtracting the expectation
    \[
    \overline{S_f}(X) = S_f(X)-\E^s S_f.
    \]
    The regularized additive functional may be extended by continuity to functions $f\in L_\infty(\R)\cap L_2(\R)$. In case $f$ is not absolutely integrable $\overline{S}_f$ is $\P^s$-almost surely equal to
    \[
    \overline{S_f}(X) = \lim_{n\to\infty}\left(\sum_{\substack{x\in X \\ \abs{x}\le k_n}}f(x) - \int_{-k_n}^{k_n} K^s(x, x)f(x)dx\right),
    \]
    for some depending on $f$ sequence satisfying $k_n\to \infty$ as $n\to\infty$ (see Definition \ref{3:reg_def}).
    
    Recall that $p$-Sobolev space is a Hilbert space of functions with the norm
    \[
    \norm{f}_{H_p} = \norm{f}_{L_2}+\normH{p}{f}, \quad \normH{p}{f}^2 = \int_\R\abs{\lambda}^{2p}\abs{\hat{f}(\lambda)}^2d\lambda.
    \]
    Here and subsequently we adopt the following convention for the Fourier transform
    \[
    \hat{f}(\omega)= \frac{1}{\sqrt{2\pi}}\F f(\omega)=\frac{1}{2\pi}\int_\R e^{-i\omega x}f(x)dx.
    \]
    \begin{theorem}\label{1:mult_formula}
        For any $f\in H_2(\R)$ we have
        \begin{equation}\label{1_eq:mult_formula}
            \E^se^{\overline{S}_f} = \exp\left(\int_{\R_+}\lambda\hat{f}(\lambda)\hat{f}(-\lambda)d\lambda\right)Q(f),
        \end{equation}
        where for some independent of $f$ constant $C$ we have
        \begin{align*}
        &\abs{Q(f)-1}\le CL(f)e^{CL(f)},\\
        &L(f) = (\normH{3/4}{f}+ \normH{2}{f})e^{8\norm{\hat{f}}_{L_1}}(\normH{3/4}{f}+ \normH{2}{f} + 1).
        \end{align*}
    \end{theorem}
    
    Having the estimate for the convergence of Laplace transform, it is possible to establish the convergence in the Kolmogorov-Smirnov metric using the Feller smoothing estimate \cite[p.~538]{F_66} (see Section \ref{sect:corr_proof}). Denote 
    \[
    F_R(x) = \P^s(\overline{S}_{f(x/R)}\le x),\quad F_{\mathcal{N}}(x)=\frac{1}{\sqrt{2\pi}}\int_{-\infty}^xe^{-\frac{t^2}{2}}dt.
    \]
    \begin{corollary}\label{1:KS_conv}
    Let $f\in H_2(\R)$ be a real-valued function. Normalize it so that
    \[
    \int_{\R_+}\lambda\hat{f}(\lambda)\hat{f}(-\lambda)d\lambda=1/2
    \]
    holds. Then we have that there exists a constant $C$ such that
    \[
    \sup_{x\in\R}\abs{F_R(x)-F_{\mathcal{N}}(x)}\le\frac{C}{\ln R}.
    \]
    \end{corollary}
    
    \subsection{Exact formula for $Q(f)$ in terms of a Fredholm determinant}
     \begin{notation_remark}
    Here and subsequently for a kernel $K(x, y)$ we denote the respective operator by $K$. Further, for a function $f\in L_\infty(\R)$ let $f$ also stand for the respective operator of pointwise multiplication on $L_2(\R)$. For a subset $A\subset\R$ by $\I_A$ we denote the indicator function of $A$. Let $\I_\pm = \I_{\R_\pm}$.
    \end{notation_remark}
    Introduce the following function
    \[
    \CT(x) = \frac{e^{-ix}}{\sqrt{2\pi}}\psi(x)\rho(x)Z_s(x)e^{i\pi\Re s\sign x}, \quad \psi(x)=e^{-\frac{i\pi}{2}\Re s\sign x}\abs{x}^{-i\Im s}.
    \]
    For a compactly supported bounded Borel $h$ define an integral transform
    \[
    \CT h(\omega) = \int_\R \CT(\omega x)h(x)dx.
    \]
    In \cite[Theorem~1.1]{G_25} it is shown that $\CT$ extends by continuity to a unitary operator (see Section~\ref{sect:diagonalization}, Theorem \ref{3:Diagonalization}). Define
    \[
    \CO{f} = \I_+\CT f\CT^*\I_+, \quad W_f = \I_+\F f\F^*\I_+.
    \]
    
    For any function $f\in \F^*L_1(\R)$ with an absolutely integrable Fourier transform denote $f = f_++f_-$ its decomposition into positive and negative frequencies $f_\pm\in\F^*L_1(\R_\pm)$. Recall that $\F^*L_1(\R)$, $\F^*L_1(\R_\pm)$ are Banach algebras with pointwise multiplication and $\norm{\hat{\cdot}}_{L_1}$-norm.
    \begin{theorem}\label{1:remainder_formula}
    Under requirements of Theorem \ref{1:mult_formula} we have
    \begin{equation}\label{1_eq:remainder_formula}
        Q(f) = \det(\mathbb{I}_{[1, \infty]}W_{e^{f_-}}\CO{e^{-f_+}}\CO{e^{-f_-}}W_{e^{f_+}}\mathbb{I}_{[1, \infty)})\exp\left(\Tr\mathbb{I}_{[1, \infty)}(\CO{f}-W_f)\mathbb{I}_{[1, \infty)}\right).
        \end{equation}
    \end{theorem}
    \subsection{Related work}
    
    \textbf{The case $s=0$}
    
    Recall that for $s=0$ the process reduces to the sine-process --- a determinantal point process induced by the kernel
    \[
    K_{\mathcal{S}}(x, y) = \frac{\sin (\pi(x-y))}{\pi(x-y)}.
    \]
    For this process the convergence in distribution of additive functionals has been shown by Soshnikov~\cite{S_00CLT}. His method does not involve direct work with the Fredholm determinants, though it is possible to calculate their asymptotics, which was done by Katz and Akhiezer \cite[Section 10.13]{BS_06}. Proposition \ref{2:mult_formula} for $s=0$ yields that
    \[
    \E^0e^{\overline{S}_f}=e^{-\hat{f}(0)}\det(\I_{[0, 1]}W_{e^f}\I_{[0, 1]}).
    \]
    Recall that Katz and Akhiezer showed that under certain conditions on $f$ we have
    \[
    \det(\I_{[0, 1]}W_{e^{f(\cdot/R)}}\I_{[0, 1]}) \sim \exp\left(R\hat{f}(0)+\int_0^\infty \lambda\hat{f}(\lambda)\hat{f}(-\lambda)d\lambda\right), \quad\text{ as }R\to\infty.
    \]
    This yields the result of Soshnikov. An exact formula for these determinants is due to Basor and Chen \cite{BC_03}. We note that Theorem \ref{1:remainder_formula} coincides with their result if $s=0$. The approach of Basor and Chen is similar to ours and is based on the factorization of Wiener-Hopf operators and the Widom's formula \cite{W_82}. See Theorem \ref{4:WH_properties} for the generalization of these statements to an arbitrary $s$. Their formula was then used by Bufetov \cite{B_25} to derive an estimate for the Kolmogorov-Smirnov distance. We note that, unlike Corollary \ref{1:KS_conv}, the latter result includes the estimate by a constant times $1/R$ for holomorphic in a strip functions. Another proof for the exact identity under weaker assumptions is due to Bufetov~\cite{B_24}.
    
    \textbf{Limit theorems for different determinantal point processes}
    
    A result, similar to Theorem \ref{1:remainder_formula}, has been derived by the author for the Bessel kernel determinantal point process \cite{G_24} (see \cite{TW_94} for the details about the process). The proof of the latter is in many ways similar to the one presented here. Both are inspired by a series of works of Basor, Widom, Ehrhardt, Chen and B\"oettcher~\cite{B_97, BC_03, BE_03, BE_03B, BEW_03, BW_00, B_02, E_03SZ}. In particular, the limit theorem for the Bessel kernel determinantal point process was first proved by Basor \cite{B_97} and then by Basor and Ehrhardt \cite{BE_03} under less restrictive assumptions. The latter proof is based on the algebraic method of Ehrhardt \cite{E_03SZ}. Special case of this result includes determinants of sums of Wiener-Hopf and Hankel operators, for which a precise formula has been obtained by Basor, Widom and Ehrhardt \cite{BEW_03}. We also note that the limit theorem for the determinantal point process with the Airy kernel \cite{TW_93} has been obtained by Basor and Widom \cite{BW_99}.
    
    \textbf{Discrete counterparts and scaling limits}
    
    The problem of calculating determinants of Wiener-Hopf determinants has a discrete analogue, which is the problem of calculating Toeplitz determinants. The latter express expectations of multiplicative functionals under the radial part of the Haar measure on the unitary group, which follows from the Weyl integration formula and the Szeg\"o-Heine formula \cite[Theorem 1.5.13]{S_05OPUC}. A direct connection between Toeplitz determinants and Wiener-Hopf determinants is expressed by the convergence of the circular unitary ensemble to the sine process. It was used by Bufetov in~\cite{B_24} to derive the exact formula for Wiener-Hopf determinants from the Borodin-Okounkov formula~\cite{BO_00} for the Toeplitz determinants. Originally obtained by Borodin and Okounkov from the Gessel Theorem and the determinantal structure of the Schur measures, the Borodin-Okounkov formula was derived more directly using operator-theoretic methods \cite{BW_00, B_02}. Calculation of determinants of sums of Toeplitz and Hankel matrices may be approached both by operator-theoretic methods \cite{BE_03} and by representation-theoretic methods \cite{B_18}. The limit theorem for the Toeplitz determinants is the statement of the Szeg\"o Theorem \cite[Section 10.4]{BS_06}.
    
    Determinants, connected with multiplicative functionals under $\P^s$, have discrete analogue for an arbitrary parameter $s$. By a result of Bourgade, Nikeghbali and Rouault the measure $\P^s$ is a scaling limit of the circular Pseudo-Jacobian ensemble \cite{BNR_06}. Multiplicative functionals of the latter may be expressed as relations of singular Toeplitz determinants. It is, therefore, interesting if Theorem \ref{1:remainder_formula} has a discrete analogue. We note that an exact formula for these relations was derived by Ehrhardt in \cite{E_97}. However, we were not able to derive its scaling limit as was done in \cite{B_24}. A more general situation of the Toeplitz determinants with a number of Fisher-Hartwig singularities was considered by Deift, Its and Krasovsky \cite{DIK_11} using the Deift-Zhou steepest descent method for the Riemann-Hilbert problem \cite{DZ_93}, but only an asymptotic was derived. It is also interesting if the connection of multiplicative functionals under $\P^s$ with singular Wiener-Hopf operators exists as in the discrete case. Asymptotics of the latter were considered by Basor and Ehrhardt for one singularity \cite{BE_04} and by Kozlowski in a more general case \cite{K_22}.
    
    \textbf{Constructions of the confluent hypergeometric point processes}
    
    Last, let us mention several constructions of the determinantal point process with the confluent hypergeometric kernel. It was first shown by Borodin and Olshanski \cite{BO_01} that $\P^s$ describes decomposition of the Hua-Pickrell measures on semi-infinite Hermitian matrices into ergodic components with respect to the action of the infinite unitary group by conjugation. They also showed that $\P^s$ is the scaling limit of the Pseudo-Jacobian orthogonal-polynomial ensemble on $\R$. A similar construction involving orthogonal polynomials on the unit circle was described by Bourgade, Nickeghbali and Rouault \cite{BNR_06}. The confluent hypergeometric kernel point process may also be constructed by Palm measures, which was shown by Bufetov \cite{B_23}. In particular, the Palm measure of $\P^s$ in zero is $\P^{s+1}$. That is, for natural $s$ the point process is $s$-th Palm measure of the sine process in zero. Last, the measure $\P^s$ appears as a degeneration of the ${}_2F_1$-process, which was shown by Borodin and Deift \cite{BD_01}.
    
    \section{Acknowledgements}
The authors are winners of the BASIS Foundation Competition and are deeply grateful to the Jury
and the sponsors.
    
    \section{Outline of proof}\label{sect:outline}
    We derive Theorems \ref{1:mult_formula}, \ref{1:remainder_formula} as statements for the expression
    \[
    \det(\I_{[0, 1]}\CO{e^f}\I_{[0, 1]}).
    \]
    The connection with the measure $\P^s$ is expressed via the following proposition.
    \begin{proposition}\label{2:mult_formula}
    For any $f\in L_1(\R)\cap L_\infty(\R)$ we have
    \[
    \E^se^{\overline{S}_f} = e^{-\Tr \I_{[0, 1]}G_f\I_{[0, 1]}}\det(\I_{[0, 1]}G_{e^f}\I_{[0, 1]}).
    \]
\end{proposition}

The proof of Theorem \ref{1:remainder_formula} is based on the Wiener-Hopf factorization for the operator $G_f$. We recall the factorization in Section \ref{sect:diagonalization} (see Theorem \ref{4:WH_properties}). Using the latter we are able to derive Theorem \ref{1:remainder_formula} first for a small class of functions.
\begin{lemma}\label{2:form_first}
We have that Theorem \ref{1:remainder_formula} holds for $f\in H_{1/2}(\R)\cap \F^*L_1(\R)\cap L_1(\R)$ such that $\mathcal{R}_f = G_f-W_f\in\mathcal{J}_1$.
\end{lemma}

Next we establish that the class of functions satisfying the restriction is sufficiently large to further extend the formula by continuity. In Section \ref{sect:tr_class} we prove the following lemma.
\begin{lemma}\label{2:diff_tr_est}
\begin{itemize}
\item We have that $\mathcal{R}_f\in\mathcal{J}_1$ for $f\in H_2(\R)$ such that $x^2f(x)\in L_2(\R)$.
\item The following estimate holds for some independent of $f$ constant
\[
\norm{\I_{[1, +\infty)}\mathcal{R}_f\I_{[1, +\infty)}}_{\mathcal{J}_1} \le C(\normH{3/4}{f}+\normH{2}{f}).
\]
\end{itemize}
\end{lemma}

Lemmata \ref{2:form_first}, \ref{2:diff_tr_est} yield that Theorem \ref{1:remainder_formula} holds for $f\in H_{1/2}\cap \F^*L_1(\R)$ such that $x^2f(x)\in L_2(\R)$. The space of such functions is dense in $H_2(\R)$. Thereby in order to show that Theorem \ref{1:remainder_formula} holds for any $f\in H_2(\R)$ it is sufficient to establish continuity of all the expressions involved with respect to the $\norm{\cdot}_{H_2}$-norm. We devote Section \ref{sect:fredholm} to the proof of the continuity of the Laplace transform of regularized additive functionals.
\begin{lemma}\label{2:reg_ext}
We have that the map $\overline{S}_-:L_1(\R)\cap L_\infty(\R)\to L_2(\P^s, \Conf(\R))$ extends by continuity to a map $L_2(\R)\cap L_\infty(\R)\to L_1(\P^s, \Conf(\R))$. Further, we have that $\E^se^{\overline{S}_f}$ is continuous with respect to $\norm{\cdot}_{L_2}+\norm{\cdot}_{L_\infty}$-norm.
\end{lemma}
We note that $\norm{\cdot}_{L_2}+\norm{\cdot}_{L_\infty} \le C\norm{\cdot}_{H_2}$ for some constant $C$.

The continuity of the exponential factor in $Q(f)$ follows from the second claim of Lemma \ref{2:diff_tr_est} since for a trace class operator $K$ the inequality $\abs{\Tr K}\le \norm{K}_{\mathcal{J}_1}$ holds. To show the continuity of the determinant we again employ the factorization properties of $G_-$, presented in Section \ref{sect:diagonalization}. We devote Section \ref{sect:rem} to the following statement.
\begin{lemma}\label{2:remainder_continuity}
We have that $Q(f)$ is continuous with respect to $\norm{\cdot}_{H_2}$-norm. 
 There exists some independent of $f$ constant $C$ such that
        \begin{align*}
        &\abs{Q(f)-1}\le CL(f)e^{CL(f)},\\
        &L(f) = (\normH{3/4}{f}+ \normH{2}{f})e^{8\norm{\hat{f}}_{L_1}}(\normH{3/4}{f}+ \normH{2}{f} + 1).
        \end{align*}
\end{lemma}
Since the integral
\[
\int_0^\infty \omega\hat{f}(\omega)\hat{f}(-\omega)d\omega
\]
is continuous with respect to $\norm{f}_{H_2}$, all statements above imply Theorem \ref{1:remainder_formula}. Theorem \ref{1:mult_formula} follows from the inequality in Lemma \ref{2:remainder_continuity}. We conclude the proof of Corollary \ref{1:KS_conv} in Section~\ref{sect:corr_proof}.

\subsection{Structure of the paper}
The paper has the following structure. In Section \ref{sect:diagonalization} we recall the connection between the kernel $K^s$ and the transform $\CT$, introduced in the beginning. Further in Section \ref{sect:fredholm} we recall basic definitions and explain the connection between multiplicative functionals and Fredholm determinants, concluding the proof of Proposition \ref{2:mult_formula}. In the second part of Section \ref{sect:fredholm} we introduce the regularization of additive functionals and prove Lemma \ref{2:reg_ext}. In Section \ref{sect:form_proof} we use Theorem \ref{4:WH_properties} to prove Lemma \ref{2:form_first}. We devote Section \ref{sect:tr_class} to the analysis of the difference $\mathcal{R}_f = G_f-W_f$ and proof of Lemma \ref{2:diff_tr_est}. In Section \ref{sect:rem} we prove Lemma \ref{2:remainder_continuity}. We conclude the paper with Section \ref{sect:corr_proof}, where we deduce Corollary \ref{1:KS_conv} from Theorem \ref{1:mult_formula}.

\section{Background on the integral transform $\CT$}\label{sect:diagonalization}
Recall the notation
\begin{equation}\label{3:T_trans_form}
\begin{aligned}
&\CT(x) = \frac{e^{-ix}}{\sqrt{2\pi}}\psi(x)\rho(x)Z_s(x)e^{i\pi\Re s \sign x},\\
&\rho(x)=\abs{x}^{\Re s}e^{-\frac{\pi}{2}\Im s\sign x}, \quad \psi(x)=e^{-\frac{i\pi}{2}\Re s\sign x}\abs{x}^{-i\Im s}.
\end{aligned}
\end{equation}
For a compactly supported Borel function $h$ define the integral transform
\[
(\CT h)(\omega) = \int_\R\CT(\omega x)h(x)dx.
\]
\begin{theorem}[{\cite[Theorem~1.1]{G_25}}]\label{3:Diagonalization}
        The operator $\CT$ defines an isometry on the dense subset $L_1(\R)\cap L_\infty(\R)\subset L_2(\R)$ and extends to a unitary operator, diagonalizing the confluent hypergeometric kernel
        \[
         \CT^*\I_{[0, 1]}\CT = \phi K^s\phi^*,
        \]
        where $\phi(x) = e^{-\frac{i\pi}{2}\Re s\sign x}\abs{x}^{i\Im s}$.
        The equality may be treated as a relation between the corresponding kernels
        \[
        \int_0^1\overline{\CT(xt)}\CT(yt)dt = \phi(x)K^s(x, y)\overline{\phi(y)}.
        \]
 \end{theorem}
 \begin{remark}
 For $s=0$ we have $\F=\mathcal{T}_0$. In \cite{G_25} a different definition for $\CT$ is given. It differs from the formula \eqref{3:T_trans_form} as follows
 \[
 \tilde{\mathcal{T}}_s(x) = \CT(x)\abs{x}^{2i\Im s}.
 \]
The choice in the formula \eqref{3:T_trans_form} is more convenient for the purposes of the paper. Apparently, unitarities of both transforms are equivalent. The statements of this section follow from the ones given in \cite{G_25} by a direct substitution. 
 \end{remark}
 Recall that for a bounded Borel $f\in L_\infty(\R)$ we defined a bounded operator
 \[
 G_f = \I_+\CT f\CT^*\I_+.
 \]
 Let $\F^*L_1(\R)$ be image of $L_1(\R)$ under the Fourier transform. By the Young inequality it is a Banach algebra with pointwise multiplication and $\norm{\hat{\cdot}}_{L_1}$-norm. It may be decomposed $\F^*L_1(\R)\simeq \F^*L_1(\R_+)\oplus \F^*L_1(\R_-)$ into subalgebras of functions with support of the Fourier transform on $\R_\pm$.
 We have the following factorization properties for the introduced operator.
 \begin{theorem}[{\cite[Corollary~1.4]{G_25}}]\label{4:WH_properties}
\begin{itemize}
    \item For any $f, g\in \F^*L_1(\R_\pm)$ we have that $G_{fg} = G_fG_g$. Further, for $f_\pm\in\F^*L_1(\R_\pm)$ we have that $G_{f_+f_-} = G_{f_-}G_{f_+}$.
    \item For a function $f\in H_{1/2}(\R)\cap\F^*L_1(\R)$ we have that $[G_{f_-}, G_{f_+}]$ is trace class and its trace is
    \[
    \Tr[G_{f_-}, G_{f_+}] = \int_0^\infty \omega \hat{f}(\omega)\hat{f}(-\omega)d\omega.
    \]
    \item Let now $s=0$, so that $G_f=W_f$ is a Wiener-Hopf operator. For $f_\pm \in \F^*L_1(\R_\pm)$ we have
    \[
    \I_{[0, 1]}W_{f_+}\I_{[1, +\infty)} = 0, \qquad \I_{[1, +\infty)}W_{f_-}\I_{[0, 1]}=0.
    \]
    \end{itemize}
\end{theorem}

These factorization properties follow from a stronger fact, connecting the Paley-Wiener spaces with the spaces of functions, whose image under $\CT$ is supported on $\R_\pm$.
\begin{theorem}[{\cite[Theorem~1.3]{G_25}}]\label{3:PW_th}
We have
\[
\CT^*\I_+\CT = \abs{x}^{2i\Im s}\F^*\I_+\F\abs{x}^{-2i\Im s}.
\]
\end{theorem}
\section{Multiplicative functionals and Fredholm determinants}\label{sect:fredholm}
\subsection{Basic definitions and proof of Proposition \ref{2:mult_formula}}
Recall that a configuration on $\R$ is a discrete subset without accumulation points. We denote the space of configurations by $\Conf(\R)$. Endow $\Conf(\R)$ with topology induced by action of $X\in\Conf(\R)$ on bounded Borel compactly supported functions $h$ by the formula $X(h) = \sum_{x\in X}h(x)$. A point process is a probability measure on $\Conf(\R)$ with the corresponding Borel $\sigma$-algebra.

For a compactly supported bounded Borel function $g$ define multiplicative and additive functional by the formulae
\[
\begin{aligned}
&\Psi_{1+g}(X) = \prod_{x\in X}(1+g(x)),\\
&S_g(X) = \sum_{x\in X}g(x),
\end{aligned}
\quad X\in \Conf(\R).
\]
In order to define how operators induce measures on $\Conf(\R)$ we first define what is locally trace class operator.

\begin{definition}\label{B:lt_def}
	A bounded operator $K$ is locally trace class if for any compact $B\subset\R$ we have that the operator $\I_BK\I_B$ is trace class.
\end{definition}

A point process $\P_K$ is determinantal if there exists a locally trace class operator such that for any compactly supported on $B$ bounded Borel function $g$ we have
\[
\E_K\Psi_{1+g} = \det(I+gK\I_B).
\]
We note that this condition defines the measure uniquely. However, different kernels may induce the same point process. The Macchi-Soshnikov Theorem \cite{M_75, S_00} asserts any locally trace class self-adjoint positive contraction of $L_2(\R)$ induces a determinantal poiint process.

Our proof of Proposition \ref{2:mult_formula} follows \cite[Section~4]{G_24} and Bufetov~\cite[Section~2.5]{B_19}. The argument is based on the regularization of the Fredholm determinant, which involves calculating a trace of an operator by integrating the diagonal. General case when such calculation is possible is covered by Mercer's theorem.

\begin{theorem}[Mercer Theorem, {\cite[Theorem~3.11.9]{S_15}}]\label{B:merc_th}
Let a continuous function $K(x, y)$ on $[a, b]^2$ induce a positive self-adjoint operator $K$ on $L_2([a, b])$. Then the operator is trace class. Further, we have
\begin{equation}\label{B_eq:diag_tr}
	\Tr K = \int_a^b K(x, x)dx.
\end{equation}
\end{theorem}

We will refer to formula \eqref{B_eq:diag_tr} as to calculation of trace over the diagonal. The formula may be extended to a more general case. Recall, however, that the kernel of an operator is defined up to a zero measure subsets. Since the measure of diagonal is zero, it is necessary to choose the kernel in a certain way, not changing the induced operator. Existence of such choice is described by the following lemma.\begin{lemma}\label{3:tr_ext}
Let $K$ be a self-adjoint trace class operator on $L_2(X)$, where $X\subset \R$ is measurable. Then there exists a kernel $K(x, y)$ of $K$ such that
\begin{equation}\label{3_eq:diag_form_ext}
\Tr K = \int_X K(x, x)dx.
\end{equation}
Further, for any $f\in L_\infty(X)$ we have
\begin{equation}\label{3_eq:diag_form_mult}
\Tr fK = \int_X f(x)K(x, x)dx.
\end{equation}
If the operator is positive, then a continuous $K(x, y)$ satisfies both equations.
\end{lemma}
\begin{proof}
Denote the spectral decomposition of $K$ by
\[
K = \sum_{\lambda_j\in\sigma(K)}\lambda_j \psi_j\langle \psi_j, -\rangle.
\]
We may then take the kernel to be
\[
K(x, y) = \sum_{\lambda_j\in\sigma(K)}\lambda_j \psi_j(x)\overline{\psi_j(y)},
\]
where $\sum_{\lambda_j\in\sigma(K)}\abs{\lambda_j}\abs{\psi_j(x)}^2$ converges almost surely on $X$ since $K$ is trace class. Trace formulae follow from the definition of trace.

If $K$ is positive and $K(x, y)$ is continuous then Theorem \ref{B:merc_th} implies the formula \eqref{3_eq:diag_form_ext}. If a function $f\in L_\infty(X)$ is positive and continuous then we apply Theorem \ref{B:merc_th} a bit differently
\[
\Tr(fK) = \Tr(\sqrt{f}K\sqrt{f}) = \int_Xf(x)K(x, x)dx.
\]
The equality may be extended by continuity to all positive $L_\infty(\R)$ functions. Last, any function may be represented as a sum of four positive $L_\infty(\R)$ functions multiplied by complex coefficients. This proves the equality \eqref{3_eq:diag_form_mult}.
\end{proof}

Let $\Pi$ be a locally trace class orthogonal projector on $L_2(\R)$. Assume that the kernel $\Pi(x, y)$ coincides with the spectral decomposition locally. Denote $d\mu_\Pi(x) = \Pi(x, x)dx$.
\begin{proposition}[{\cite[Proposition~5.2]{G_24}}]\label{3:mult_cont}
	For any $f\in L_1(\R, d\mu_\Pi)\cap L_\infty(\R)$ we have
	\begin{equation}\label{3_eq:mult_formula}
	\E_K\Psi_{1+f} = \det{}_2(I+f\Pi)e^{\int_\R \Pi(x, x)f(x)dx}.
	\end{equation}
	Multiplicative functionals induce a continuous mapping
	\[
	\Psi_{1+(-)}:f\in L_1(\R, d\mu_\Pi)\cap L_\infty(\R, d\mu_\Pi) \to L_1(\Conf(\R), \P_\Pi).
	\]
\end{proposition}
\begin{remark}
The right-hand side of the expression \eqref{3_eq:mult_formula} is equal to the determinant of $I+f\Pi\I_B$, $B=\supp f$ if the function is bounded compactly supported. In this case by Lemma \ref{3:tr_ext} the expression is the definition of the regularized determinant.
\end{remark}

We note that Proposition \ref{3:mult_cont} is applicable to the kernel $K^s(x, y)$ defined in the formula \eqref{1_eq:CHK_def}, since it is continuous everywhere except zero. Let us show how Proposition \ref{3:mult_cont} implies Proposition \ref{2:mult_formula}.
\begin{proof}[Proof of Proposition \ref{2:mult_formula}]
Since 
\[
L_1(\R, dx)\cap L_\infty(\R, dx)\subset L_1(\R, K^s(x, x)dx)\cap L_\infty(\R, K^s(x, x)dx)
\]
by Proposition \ref{3:mult_cont} it is sufficient to establish that for any $f\in L_1(\R)\cap L_\infty(\R)$ the identities
\begin{equation}\label{3_eq:det_reg_equality}
\begin{aligned}
&\det{}_2(I+fK^s) = \det{}_2(I+\I_{[0, 1]}G_f\I_{[0, 1]}),\\
&\int_\R f(x)K^s(x, x)dx = \int_0^1 G_f(x, x)dx
\end{aligned}
\end{equation}
hold. By $G_f(x, y)$ we denoted the kernel of the operator $G_f$. Let us show how the assertion follows. It will be shown below that the operator $\I_{[0, 1]}G_f\I_{[0, 1]}$ is trace-class under the above assumption. Further, its trace may be calculated over the diagonal. So we have by the definition of the regularized determinant
\[
\det(I+\I_{[0, 1]} G_f\I_{[0, 1]}) = \det{}_2(I+\I_{[0, 1]}G_f\I_{[0, 1]})\exp\left(\int_0^1 G_f(x, x)dx\right).
\]
Therefore the expressions \eqref{3_eq:det_reg_equality} and the identity \eqref{3_eq:mult_formula} yield that
\[
\det(I+\I_{[0, 1]}G_f\I_{[0, 1]}) = \E^s\Psi_{1+f},
\]
which is the desired claim, once $f=e^{g}-1$ is substituted. In the latter case we have $\Psi_{1+f} = \Psi_{e^g}=e^{S_g}$. Last, we recall that
\[
\E^s S_f = \int_\R K^s(x, x)f(x)dx.
\]
As was noted above, by the second equality \eqref{3_eq:det_reg_equality} the right-hand side equals $\Tr(\I_{[0, 1]}G_f\I_{[0, 1]})$, which concludes the proof by a direct substitution.

We proceed to the proof of identities \eqref{3_eq:det_reg_equality}. First by Theorem \ref{3:Diagonalization} we have that the integrals of diagonals coincide
\[
\int_\R f(x)K^s(x, x)dx = \int_\R f(x)\int_0^1 \abs{\CT(xt)}^2dt = \int_0^1dt\int_\R\CT(tx)f(x)\CT^*(tx)dx = \int_0^1\CO{f}(t, t)dt,
\]
where $G_f(x, y)$ stands for the kernel of the operator $G_f$. We note that the kernel $G_f(x, y)$ is continuous under our assumptions on $f$. Further, $G_f$ is positive if the function $f$ is positive. Decomposing the function and, consequently, the respective operator into the sum of four positive continuous functions with complex coefficients, we conclude by Theorem \ref{B:merc_th} that $\I_{[0, 1]}G_f\I_{[0, 1]}$ is trace class and that
\[
\int_0^1\CO{f}(t, t)dt = \Tr(\I_{[0, 1]}\CO{f}\I_{[0, 1]})
\]
holds.

Next let us show that
\[
\det{}_2(I+fK^s) = \det{}_2(I+\I_{[0, 1]}\CO{f}\I_{[0, 1]}).
\]
Using unitary equivalence and conjugating by $\CT$ one derives
\[
\det{}_2(I+f\CT^*\I_{[0, 1]}\CT) = \det{}_2(I+\CT f\CT^*\I_{[0, 1]}),
\]
where, since $\det{}_2(I+AB) = \det{}_2(I+BA)$, we have
\[
\det{}_2(I+\CT f\CT^*\I_{[0, 1]}) = \det{}_2(I+\I_{[0,1]}\CO{f}\I_{[0, 1]}).
\]
 
 Proposition \ref{2:mult_formula} is proved.
\end{proof}

\subsection{Regularization of additive functionals and proof of Lemma \ref{2:reg_ext}}

By Proposition \ref{3:mult_cont} multiplicative functionals define almost surely finite functions on $\Conf(\R)$ only for bounded absolutely integrable functions. However, it is possible to extend their definition to a wider class of functions. Let us give an example. Recall that the sine process is translationally invariant, which yields that the sum
\[
\sum_{x\in X}\frac{1}{i+x}, \quad X\in \Conf(\R)
\]
almost surely does not converge absolutely under the sine process measure. However, the limit
\[
\lim_{n\to\infty}\sum_{x\in X, \abs{x}\le n}\frac{1}{i+x}, \quad X\in \Conf(\R)
\]
almost surely exists. Formally, despite the additive functional $S_{1/(x+i)}$ is not defined, the limit $S_{\I_{[-n, n]}/(x+i)}$ as $n\to \infty$ exists. Regularization of functionals was introduced by Bufetov \cite{B_18Q, B_19} and then applied to construction of Radon-Nikodym derivatives between Palm measures of determinantal point process (see \cite{B_18Q}).

We proceed to the precise construction.
Assume $\Pi$ is a locally trace class orthogonal projection operator on $L_2(\R)$ and let $\P_\Pi$ be the respective determinantal measure. As above we also assume that the kernel $\Pi(x, y)$ is chosen in the sense of Lemma \ref{3:tr_ext} locally. We note that the kernel $K^s(x, y)$ satisfies this condition. Recall the notation $d\mu_\Pi = \Pi(x, x)dx$.

First for $f\in L_1(\R, d\mu_\Pi)\cap L_\infty(\R)$ define the regularized additive functional
\[
\overline{S}_f = S_f-\E_\Pi S_f.
\]
It may be shown directly via correlation functions that its variance is
\[
\E_\Pi\abs*{\overline{S}_f}^2-\abs*{\E_\Pi \overline{S}_f}^2 = \int_{\R^2}\abs{f(x)-f(y)}^2\abs{\Pi(x, y)}^2dxdy.
\] 
Self-adjointness of $\Pi$ and the projector property $\Pi^2=\Pi$ yield
\begin{equation}\label{3_eq:projector_integral}
\int_\R\Pi(x, y)\overline{\Pi(x, y)}dy = \int_\R\Pi(x, y)\Pi(y, x)dy = \Pi(x, x),
\end{equation}
from which we deduce
\[
\text{Var } \overline{S}_f \le 2\norm{f}_{L_2(\R, d\mu_\Pi)}.
\]
Thereby the map
\[
\overline{S}_-:L_1(\R, d\mu_\Pi)\cap L_\infty(\R) \to L_2(\Conf(\R), \P_\Pi).
\]
may be extended by continuity to a map on $L_2(\R, d\mu_\Pi)\cap L_\infty(\R)$.

\begin{definition}\label{3:reg_def}
For a measure $\P_\Pi$ a regularized additive functional $\overline{S}_f$ is a map 
$$L_2(\R, d\mu_\Pi)\cap L_\infty(\R) \to L_2(\Conf(\R), \P^s),$$
obtained as an extension by continuity of $S_f-\E S_f$ from the dense subset of bounded Borel compactly supported functions. Since any such function $f\in L_2(\R, d\mu_\Pi)\cap L_\infty(\R)$ may be approximated in $\norm{\cdot}_{L_2(\R, d\mu_\Pi)}+\norm{\cdot}_{L_\infty}$-norm by the sequence $\I_{[-n, n]}f$, it is $\P_\Pi$-a. s. equal to
\[
\overline{S}_f(X) = \lim_{k\to\infty} \left(\sum_{x\in X, \abs{x}\le n_k}f(x) - \int_{-n_k}^{n_k} f(x)\Pi(x, x)dx\right)
\]
for some subsequence $n_k$, possibly depending on $f$.
\end{definition}

In order to prove Proposition \ref{2:reg_ext} one needs continuity of the Laplace transform.
\begin{lemma}\label{3:laplace_cont}
We have that $\E_\Pi^{\overline{S}_f}$ is continuous with respect to the $\norm{f}_{L_2(\R, d\mu_\Pi)}+ \norm{f}_{L_\infty}$-norm.
\end{lemma}
\begin{proof}
By Proposition \ref{3:mult_cont} we have
\[
\E_\Pi^{\overline{S}_f}=\det{}_2(I+(e^f-1)\Pi)e^{\int_\R (e^{f(x)}-1-f(x))\Pi(x, x)dx}.
\]

Let us show that the integral under the exponent is continuous with respect to the $\norm{f}_{L_2(\R, d\mu_\Pi)}+ \norm{f}_{L_\infty}$-norm. Define the entire function
\[
h(x) = \frac{e^x-1-x}{x^2}, \quad \abs{h(x)} \le e^{\abs{x}}.
\]
Let $\gamma$ be a contour of a radius $\norm{f}_{L_\infty}+1$ encircling zero. We have
\[
e^{f(x)}-1-f(x) = \int_\gamma h(\lambda)\frac{f^2(x)}{f(x)-\lambda}d\lambda.
\]
The difference may then be expressed
\[
e^{f(x)}-e^{g(x)}-f(x) + g(x) = \int_{\gamma_{f, g}} h(\lambda)\frac{f(x)g(x)(f(x)-g(x)) + \lambda(g^2(x)-f^2(x))}{(f(x)-\lambda)(g(x)-\lambda)},
\]
where $\gamma_{f, g}$ is a circle around zero with the radius of $1+\norm{f}_{L_\infty}+\norm{g}_{L_\infty}$. The integral on the right-hand side may be bounded by
\begin{multline*}
\abs*{e^{f(x)}-e^{g(x)}-f(x) + g(x)} \le e^{1+\norm{f}_{L_\infty}+\norm{g}_{L_\infty}}(1+\norm{f}_{L_\infty}+\norm{g}_{L_\infty})\times\\
\times\left(\abs{f(x)g(x)(f(x)-g(x))} + \abs{g^2(x)-f^2(x)}\right).
\end{multline*}
Thereby the integral is continuous with respect to the $L_1$ norm of the right-hand side. Using the Cauchy-Bunyakovsky-Schwarz inequality we conclude
\begin{align*}
&\norm{fg(f-g)}_{L_1(\R, d\mu_\Pi)} \le \norm{f}_{L_\infty}\norm{g}_{L_2(\R, d\mu_\Pi)}\norm{f-g}_{L_2(\R, d\mu_\Pi)}, \\
&\norm{f^2-g^2}_{L_1(\R, d\mu_\Pi)} \le \norm{f-g}_{L_2(\R, d\mu_\Pi)}\norm{f+g}_{L_2(\R, d\mu_\Pi)}.
\end{align*}
This proves the claim.

To show the continuity of the regularized determinant $\det{}_2(I+(e^f-1)\Pi)$ recall that it is continuous with respect to the Hilbert-Schmidt norm of $(e^f-1)\Pi$. Observe that for $f, g\in L_1(\R, d\mu_\Pi)\cap L_\infty(\R)$ the Hilbert-Schmidt norm of the difference is
\[
\norm*{(e^f-e^g)\Pi}_{\mathcal{J}_2}^2 = \int_{\R^2}\abs{e^{f(x)}-e^{g(x)}}^2\Pi(x, y)\overline{\Pi(x, y)}dxdy,
\]
from which, substituting the equality \eqref{3_eq:projector_integral}, we deduce that
\[
\norm*{(e^f-e^g)\Pi}_{\mathcal{J}_2} = \norm{e^f-e^g}_{L_2(\R, d\mu_\Pi)} \le e^{2(\norm{f}_{L_\infty}+\norm{g}_{L_\infty})}\norm{f-g}_{L_2(\R, d\mu_\Pi)}
\]
holds. This finishes the proof.
\end{proof}

Observe that for some constant $C$ we have
\[
\norm{\cdot}_{L_2(\R, d\mu_{K^s})} + \norm{\cdot}_{L_\infty} \le C(\norm{\cdot}_{L_2}+\norm{\cdot}_{L_\infty})
\]
With this observation Proposition \ref{2:reg_ext} follows from Lemma \ref{3:laplace_cont} directly.

\section{Derivation of the formula}\label{sect:form_proof}
As soon as Theorem \ref{4:WH_properties} is established, it is possible to apply operator-theoretic methods from \cite{BC_03, BE_03B, BEW_03, B_02} to calculation of the determinant from Proposition \ref{2:mult_formula}. In particular, we refer to \cite[Sect. 3]{G_24} for a brief exposition of the proof of Theorem \ref{1:remainder_formula} for $s=0$. Under assumption $\mathcal{R}_f = G_f-W_f\in\mathcal{J}_1$ the method is similar.

The calculation is based on two classical results about Fredholm determinant. First is the Ehrhardt's generalization of the Helton-Howe formula.

\begin{theorem}[Ehrhardt, \cite{E_03}]\label{B:helton_howe}
	For bounded operators $A, B$ with a trace class commutator we have that $e^Ae^Be^{-A-B}-I$ is trace class. Further, the following formula holds
	\[
	\det e^Ae^Be^{-A-B}=e^{\frac{1}{2}\Tr[A, B]}.
	\]
\end{theorem}

In case $s=0$ Theorems \ref{B:helton_howe}, \ref{4:WH_properties} yield that
\[
\det\bra*{e^{-W_{f_+}}W_{e^f}e^{-W_{f_-}}} = \exp\bra*{\int_0^\infty \omega\hat{f}(\omega)\hat{f}(-\omega)d\omega}.
\]
This formula has been used by Basor and Ehrhardt to derive the asymptotics of the Bessel operator determinants (see \cite[Section 5.2]{BE_03}). Since $W_f$ and $G_f$ are unitarily equivalent by Theorem \ref{3:PW_th}, the formula still holds if one replaces all Wiener-Hopf operators by $G_-$. Under restriction $\mathcal{R}_f\in\mathcal{J}_1$ we show a different generalization.

\begin{lemma}\label{4:widom_formula}
Let $f\in \F^*L_1(\R)\cap H_{1/2}(\R)$ satisfy $\mathcal{R}_f \in \mathcal{J}_1$. Then we have
\[
\det(e^{-W_{f_+}}G_{e^f}e^{-W_{f_-}}) = \exp\left(\Tr\mathcal{R}_f+\int_0^\infty \omega\hat{f}(\omega)\hat{f}(-\omega)d\omega\right).
\]
\end{lemma}
\begin{proof}
First we have that
\[
\det(e^{-W_{f_+}}G_{e^f}e^{-W_{f_-}}) = \det(G_{e^f}e^{-W_{f_-}}e^{-W_{f_+}}),
\]
where substitution of Theorem \ref{4:WH_properties}, which states that $G_{e^f} = e^{G_{f_-}}e^{G_{f_+}}$, gives
\[
G_{e^f}e^{-W_{f_-}}e^{-W_{f_+}} = e^{G_{f_-}}e^{G_{f_+}}e^{-G_f}e^{G_f}e^{-W_f}e^{W_f}e^{-W_{f_-}}e^{W_{f_+}}.
\]
Thereby we may express the determinant as a product
\[
\det(e^{-W_{f_+}}G_{e^f}e^{-W_{f_-}}) = \det(e^{G_{f_-}}e^{G_{f_+}}e^{-G_f})\det(e^{G_f}e^{-W_f})\det(e^{W_f}e^{-W_{f_-}}e^{-W_{f_+}}),
\]
where the first and the last factors by Theorems \ref{B:helton_howe}, \ref{4:WH_properties} contribute
\[
\det(e^{G_{f_-}}e^{G_{f_+}}e^{-G_f})\det(e^{W_f}e^{-W_{f_-}}e^{-W_{f_+}}) = \exp\left(\int_0^\infty \omega\hat{f}(\omega)\hat{f}(-\omega)d\omega\right).
\]
Last, by Theorem \ref{B:helton_howe} and the trace class condition we have
\[
\det(e^{G_f}e^{-W_f}) = \det(e^{G_f}e^{-W_f}e^{-\mathcal{R}_f})e^{\Tr\mathcal{R}_f} = e^{\Tr\mathcal{R}_f},
\]
where we used that $\Tr [G_f, W_f] = \Tr [\mathcal{R}_f, W_f] = 0$ and $\det e^{\mathcal{R}_f}=e^{\Tr\mathcal{R}_f}$. This finishes the proof.
\end{proof}

The connection between the determinant in Lemma \ref{4:widom_formula} and the Laplace transforms of additive functionals follows from the Jacobi-Dodgeson identity.
\begin{theorem}[Jacobi-Dodgeson identity, {\cite[Proposition~6.2.9]{S_05OPUC}}]\label{B:JD}
	Let $K$ be a trace class operator. Assume $I+K$ is invertible. Let $P$ be an operator 
	of orthogonal projection and $Q=I-P$. Then we have the relation
	\[
		\det(P(I+K)P) = \det(I+K)\det(Q(I+K)^{-1}Q).
	\]
\end{theorem}

Theorem \ref{B:JD} is not applicable directly to the determinant in Proposition \ref{2:mult_formula}, since $G_{e^f-1}$ is not trace class. However, it may be expressed in terms of a determinant of the operator from Lemma \ref{4:widom_formula}, restricted to $L_2([0, 1])$. This calculation follows \cite{B_02, G_24}.

\begin{lemma}
For any $f\in L_1(\R)\cap \F^*L_1(\R)$ we have
\begin{equation}\label{4_eq:form_1}
\det(\I_{[0, 1]}G_{e^f}\I_{[0, 1]}) = e^{\hat{f}(0)}\det(\I_{[0, 1]}e^{-W_{f_+}}G_{e^f}e^{-W_{f_-}}\I_{[0, 1]}).
\end{equation}
\end{lemma}
\begin{proof}
Denote $G_{e^f}=A$. We will show that the lemma holds if one replaces $G_{e^f}$ by an arbitrary $A$ such that $\I_{[0, 1]}(A-I)\I_{[0, 1]}\in \mathcal{J}_1$.

We have
\begin{multline*}
		\det (\I_{[0, 1]}A\I_{[0, 1]}) = 
		\det(\I_{[0, 1]}W_{e^{f_+}}W_{e^{-f_+}}A
		W_{e^{-f_-}}W_{e^{f_-}}\I_{[0, 1]}) =\\
		= \det (\I_{[0, 1]}W_{e^{-f_+}}A
		W_{e^{-f_-}}\I_{[0, 1]})\det (e^{\I_{[0, 1]}W_{f_-}
			\I_{[0, 1]}}e^{\I_{[0, 1]}W_{f_+}\I_{[0, 1]}}),
	\end{multline*}
	where the second equality follows from the first and last assertions of Theorem \ref{4:WH_properties}.  It is remaining to show that
	\begin{equation}\label{4_eq:div_expr}
		\det (e^{\I_{[0, 1]}W_{f_-}\I_{[0, 1]}}
		e^{\I_{[0, 1]}W_{f_+}\I_{[0, 1]}}) = e^{\hat{f}(0)}.
	\end{equation}
	
	By Theorem \ref{B:merc_th} the operator 
	$\I_{[0, 1]}W_f\I_{[0, 1]}$ is trace class for 
	$f\in L_1(\R)\cap \F^* L_1(\R)$, so the left-hand 
	side of \eqref{4_eq:div_expr} is equal to
	\[
		\det \left(e^{\I_{[0, 1]}W_{f_-}\I_{[0, 1]}}
		e^{\I_{[0, 1]}W_{f_+}\I_{[0, 1]}}e^{-\I_{[0, 1]}
			W_f\I_{[0, 1]}}\right)e^{\Tr(\I_{[0, 1]}W_f\I_{[0, 1]})},
	\]
	where $\Tr(\I_{[0, 1]}W_f\I_{[0, 1]}) = 
	\hat{f}(0)$ holds due to Theorem \ref{B:merc_th}. A direct calculation shows that $\I_{[0, 1]}\F f$ is Hilbert-Schmidt for a square integrable $f$. Thereby the commutator $[\I_{[0, 1]}W_{f_-}\I_{[0, 1]}, \I_{[0,1]}W_{f_+}\I_{[0, 1]}]$ is trace class and its trace is zero. This concludes that the first factor in the above expression is $1$ by Theorem \ref{B:helton_howe} and implies that the equality \eqref{4_eq:div_expr} holds.
\end{proof}

\begin{proof}[Proof of Lemma \ref{2:form_first}]
By Lemma \ref{4:widom_formula} we have that Theorem \ref{B:JD} is applicable to the determinant in the right-hand side of the relation \eqref{4_eq:form_1}. Applying the Jacobi-Dodgeson identity, using Theorem \ref{4:WH_properties} and the formula in Lemma \ref{4:widom_formula} we have
\[
\det(\I_{[0, 1]}e^{-W_{f_+}}G_{e^f}e^{-W_{f_-}}\I_{[0, 1]}) = e^{\Tr\mathcal{R}_f + \int_0^\infty\omega\hat{f}(\omega)\hat{f}(-\omega)d\omega}\det(\I_{[1, \infty)}W_{e^{f_-}}G_{e^{-f_+}}G_{e^{-f_-}}W_{e^{f_+}}\I_{[1, \infty)}).
\]
To conclude the proof it is remaining to observe that
\[
\hat{f}(0) = \Tr (\I_{[0, 1]}W_f\I_{[0, 1]}), \quad \Tr(\I_{[0, 1]}\mathcal{R}_f\I_{[1, +\infty)})=0.
\]
\end{proof}

\section{Proof of Lemma \ref{2:diff_tr_est}}\label{sect:tr_class}

In this section we establish Lemma \ref{2:diff_tr_est}. Our approach follows \cite{BE_03, G_24}. We seperate the section into several parts. In the first part we recall preliminary statements: we formulate an estimate for integral operators of a certain form (Proposition \ref{B:norm_est}), derive a convenient estimate for hypergeometric functions (Lemma \ref{5:asymp_lemma}) and prove that $p$-Sobolev functions are H\"older (Lemma \ref{C:holder_cond}). In the second part we introduce a decomposition of the kernel $\mathcal{R}_f(x, y)$ (Lemma \ref{5:ker_expr}) and prove trace estimates for each component (Lemma \ref{5:tr_est}), which imply the second claim of Lemma \ref{2:diff_tr_est}. In the third part we show (Lemma \ref{5:01_est}) that $\I_{[0, 1]}\CT f\in\mathcal{J}_1$ for $f$ satisfying restrictions of Lemma \ref{2:diff_tr_est}, which yields that $\I_{[0, 1]}\mathcal{R}_f\in\mathcal{J}_1$. The latter, given the second claim proved, implies that $\mathcal{R}_f\in\mathcal{J}_1$. We conclude the proof in the last part.

\begin{notation_remark}
Here and in the following section we write $A_f\lesssim B_f$ for variables depending on a function $f$ if for some independent of $f$ constant $C$ we have $A_f\le CB_f$. Denote $f_0(x) = f(x)-f(0)$.
\end{notation_remark}

\subsection{Preliminaries}

The main statement we will use in the following subsection gives an estimate on trace norm of an integral operator of a certain form.

\begin{proposition}[{\cite[Proposition~7.1]{G_24}}]\label{B:norm_est}
For an integral operator on $L_2(X, d\mu_X)$ with the kernel
\[
K(x, y) = \int_Y h_1(x, t)h_2(y, t)a(t)d\mu_Y(t),
\]
where $h_1$, $h_2$ are measurable functions on $X\times Y$ we have
\[
\norm{K}_{\mathcal{J}_1(L_2(X, d\mu_X))}\le \int_Y\abs{a(t)}\left(\int_X\abs{h_1(x, t)}^2d\mu_X(x)\right)^{1/2}\left(\int_X\abs{h_2(x, t)}^2d\mu_X(x)\right)^{1/2}d\mu_Y(t)
\]
provided the right-hand side is finite.
\end{proposition}

Let us give a precise estimate for the kernel of the introduced operator $\CT$. Recall the definition
\[
\CT(x) = \frac{e^{-ix}}{\sqrt{2\pi}}\psi(x)\rho(x)Z_s(x)e^{i\pi\Re s \sign x}, \quad \rho(x)\psi(x)=x^{\bar{s}}e^{-\frac{i\pi}{2}\bar{s}},
\]
where the branch of power is chosen by the condition $\arg x\in [0, \pi]$. %Denote
%\[
%\mathcal{O}(x)=\frac{e^{-ix}}{\sqrt{2\pi}}\psi(x)\rho(x)Z_s(x), \quad \CT(x)= \mathcal{O}(x)e^{i\pi\Re s \sign x}.
%\]
The asymptotic expansion of the confluent hypergeometric function as $z\to\infty$ is (see \cite[13.5.1]{AS_64}):
\begin{multline}\label{5_eq:hypergeom_asymp}
\FO{a}{b}{z}=\frac{\Gamma(b)}{\Gamma(b-a)}e^{\pm i\pi a}z^{-a}\bra*{1-\frac{a(1+a-b)}{z}+O(\abs{z}^{-2})}+\\
+\frac{\Gamma(b)}{\Gamma(a)}e^zz^{a-b}\bra*{1+\frac{(b-a)(1-a)}{z}+O(\abs{z}^{-2})},
\end{multline}
where the upper sign being taken if $\arg z\in (-\pi/2, 3\pi/2)$, the lower sign if $\arg z\in (-3\pi/2, -\pi/2]$. Further, recall the derivative formula
\begin{equation}\label{5_eq:hypergeom_der}
\frac{d}{dz}\FO{a}{b}{z} = \frac{a}{b}\FO{a+1}{b+1}{z}.
\end{equation}
Denote
\[
\mathcal{D}(x) = e^{-\frac{i\pi}{2}\bar{s}}x^{\bar{s}}Z_s(x)-1,\quad \mathfrak{D}(x) = \frac{e^{-ix}}{\sqrt{2\pi}}\mathcal{D}(x).
\]
\begin{lemma}\label{5:asymp_lemma}
We have
\begin{align*}
&\sqrt{2\pi}\abs*{\mathfrak{D}(x)} = \abs*{\mathcal{D}(x)}\lesssim \frac{\abs{x}^{\Re s}}{1+\abs{x}^{1+\Re s}},\\
&\abs*{\mathfrak{D}'(x)-\frac{\abs{s}^2e^{-ix}}{\sqrt{2\pi}x}}\lesssim \frac{\abs{x}^{\Re s-1}}{1+\abs{x}^{1+\Re s}}.
\end{align*}
\end{lemma}
\begin{proof}
The lemma follows directly from the asymptotic expansion \eqref{5_eq:hypergeom_asymp} and the derivative formula \eqref{5_eq:hypergeom_der}. We present the calculations for convenience of the reader.

Substituting the expansion \eqref{5_eq:hypergeom_asymp} into the formula for $Z_s(x)$ we have as $x\to +\infty$
\begin{multline*}
Z_s(x) = \Gam{1+s}{1+2\Re s}\FO{\bar{s}}{1+2\Re s}{ix}=e^{i\pi\bar{s}}\abs{x}^{-\bar{s}}e^{-\frac{i\pi}{2}\bar{s}}\bra*{1-\frac{\abs{s}^2}{ix}+O(\abs{x}^{-2})} +\\
+\Gam{1+s}{\bar{s}}e^{ix}\abs{x}^{-1-s}e^{-(1+s)\frac{i\pi}{2}}\bra*{1+O(\abs{x}^{-1})}.
\end{multline*}
As $x\to -\infty$ we similarly have
\begin{multline*}
Z_s(x) = \Gam{1+s}{1+2\Re s}\FO{\bar{s}}{1+2\Re s}{ix}=e^{-i\pi\bar{s}}\abs{x}^{-\bar{s}}e^{\frac{i\pi}{2}\bar{s}}\bra*{1-\frac{\abs{s}^2}{ix}+O(\abs{x}^{-2})} +\\
+\Gam{1+s}{\bar{s}}e^{-ix}\abs{x}^{-1-s}e^{(1+s)\frac{i\pi}{2}}\bra*{1+O(\abs{x}^{-1})}.
\end{multline*}
We note that $e^{-\frac{i\pi}{2}\bar{s}}x^{\bar{s}} = e^{-\frac{i\pi}{2}\bar{s}\sign x}\abs{x}^{\bar{s}}$ for $\arg x \in [0, \pi]$. Substituting the above expansions for $Z_s(x)$ into the definition of $\mathcal{D}(x)$ we conclude that
\begin{equation}\label{5_eq:D_asymp}
\mathcal{D}(x) = i\frac{\abs{s}^2}{x} +\Gam{1+s}{\bar{s}}e^{ix}e^{-\frac{i\pi}{2}(1+2\Re s)\sign x}\abs{x}^{-1-2i\Im s}+ O(\abs{x}^{-2}),
\end{equation}
as $\abs{x}\to \infty$. This yields the first inequality.

Express the derivative as follows
\begin{equation}\label{5_eq:5_1_der}
\mathfrak{D}'(x)=\bra*{\frac{e^{-ix}}{\sqrt{2\pi}}\mathcal{D}(x)}' = -i\frac{e^{-ix}}{\sqrt{2\pi}}\mathcal{D}(x)+\frac{e^{-ix}}{\sqrt{2\pi}}\mathcal{D}'(x),
\end{equation}
where by virtue of the equality \eqref{5_eq:D_asymp} we have the following asymptotic for the first term
\begin{equation}\label{5_eq:5_1_1term}
-i\frac{e^{-ix}}{\sqrt{2\pi}}\mathcal{D}(x) = \frac{\abs{s}^2e^{-ix}}{\sqrt{2\pi}x} - \frac{i}{\sqrt{2\pi}}\Gam{1+s}{\bar{s}}e^{-\frac{i\pi}{2}(1+2\Re s)\sign x}\abs{x}^{-1-2i\Im s} + O(\abs{x}^{-2}).
\end{equation}

We proceed to the asymptotic for the second term in the expression \eqref{5_eq:5_1_der}. We have
\[
\mathcal{D}'(x) = \bar{s}\frac{e^{-\frac{i\pi}{2}\bar{s}}x^{\bar{s}}Z_s(x)}{x}+e^{-\frac{i\pi}{2}\bar{s}}x^{\bar{s}}Z_s'(x)=\bar{s}\frac{1+\mathcal{D}(x)}{x}+e^{-\frac{i\pi}{2}\bar{s}}x^{\bar{s}}Z_s'(x),
\]
where the first term is asymptotically $\bar{s}/x+O(\abs{x}^{-2})$ by the formula \eqref{5_eq:D_asymp}.
Applying the derivative formula $\eqref{5_eq:hypergeom_der}$ to the definition of $Z_s(x)$ we get
\[
Z_s'(x) =\Gam{1+s}{1+2\Re s}\frac{d}{dx}\left(\FO{\bar{s}}{1+2\Re s}{ix}\right) = i\Gam{1+s}{1+2\Re s}\frac{\bar{s}}{1+2\Re s}\FO{1+\bar{s}}{2+2\Re s}{ix} 
\]
Next we substitute the asymptotic expansion \eqref{5_eq:hypergeom_asymp} into the above formula to deduce that
\begin{multline}\label{5_eq:der_Z_asympt}
Z_s'(x) = i\frac{\bar{s}}{1+2\Re s}\Gam{1+s}{1+2\Re s}\bigg(\Gam{2+2\Re s}{1+s}e^{\frac{i\pi}{2}(1+\bar{s})\sign x}\abs{x}^{-1-\bar{s}}(1+O(\abs{x}^{-1}))+\\
+i\frac{\bar{s}}{1+2\Re s}\Gam{2+2\Re s}{1+\bar{s}}e^{ix}e^{-\frac{i\pi}{2}(1+s)\sign x}\abs{x}^{-1-s}(1+O(\abs{x}^{-1}))\bigg), \quad \text{ as }\abs{x}\to\infty.
\end{multline}
The property $z\Gamma(z) = \Gamma(z+1)$ implies that
\begin{equation}\label{5_eq:gam_ident}
\begin{aligned}
&\frac{1}{1+2\Re s}\Gam{1+s}{1+2\Re s}\Gam{2+2\Re s}{1+s} = \frac{\Gamma(2+2\Re s)}{(1+2\Re s)\Gamma(1+2\Re s)}=1,\\
%\frac{1}{1+2\Re s}\Gam{1+s, 2+2\Re s}{1+s, 1+2\Re s} = 1, \\
&\frac{\bar{s}}{1+2\Re s}\Gam{1+s, 2+2\Re s}{1+\bar{s}, 1+2\Re s} =\frac{\bar{s}\Gamma(\bar{s})\Gamma(1+s)\Gamma(2+2\Re s)}{\Gamma(\bar{s})\Gamma(1+\bar{s})(1+2\Re s)\Gamma(1+2\Re s)}= \Gam{1+s}{\bar{s}}.
\end{aligned}
\end{equation}
We again note that $e^{\frac{i\pi}{2}(1+\bar{s})\sign x}\abs{x}^{-1-\bar{s}} = e^{\frac{i\pi}{2}(1+\bar{s})}x^{-1-\bar{s}}$ and $e^{-\frac{i\pi}{2}\bar{s}}x^{\bar{s}} = e^{-\frac{i\pi}{2}\bar{s}\sign x}\abs{x}^{\bar{s}}$ for $\arg x \in [0, \pi]$. Formulae \eqref{5_eq:der_Z_asympt}, \eqref{5_eq:gam_ident} yield
\[
e^{-\frac{i\pi}{2}\bar{s}}x^{\bar{s}}Z_s'(x) = -\frac{\bar{s}}{x}(1+O(\abs{x}^{-1})) + i\Gam{1+s}{\bar{s}}e^{ix}e^{-\frac{i\pi}{2}(1+s+\bar{s})\sign x}\abs{x}^{1-s+\bar{s}}(1+O(\abs{x}^{-1})),
\]
as $\abs{x}\to\infty$. We conclude that
\begin{multline}\label{5_eq:5_1_2term}
\mathcal{D}'(x) = \bar{s}\frac{1+\mathcal{D}(x)}{x} + e^{-\frac{i\pi}{2}\bar{s}}x^{\bar{s}}Z_s'(x) =\frac{\bar{s}}{x} - \frac{\bar{s}}{x}+\\
+ i\Gam{1+s}{\bar{s}}e^{ix}e^{-\frac{i\pi}{2}(1+2\Re s)\sign x}\abs{x}^{-1-2i\Im s} + O(\abs{x}^{-2}),\quad \abs{x}\to\infty,
\end{multline}
where, as was noted above, the asymptotic for the first term follows from the expression \eqref{5_eq:D_asymp}.

Finally, substituting formulae \eqref{5_eq:5_1_1term}, \eqref{5_eq:5_1_2term} into the identity \eqref{5_eq:5_1_der} we conclude
\begin{multline*}
\bra*{\frac{e^{-ix}}{\sqrt{2\pi}}\mathcal{D}(x)}' = \frac{\abs{s}^2e^{-ix}}{\sqrt{2\pi}x} - \frac{i}{\sqrt{2\pi}}\Gam{1+s}{\bar{s}}e^{-\frac{i\pi}{2}(1+2\Re s)\sign x}\abs{x}^{-1-2i\Im s} +\\
+\frac{i}{\sqrt{2\pi}}\Gam{1+s}{\bar{s}}e^{-\frac{i\pi}{2}(1+2\Re s)\sign x}\abs{x}^{-1-2i\Im s}+O(\abs{x}^{-2}) = \frac{\abs{s}^2e^{-ix}}{\sqrt{2\pi}x} + O(\abs{x}^{-2}), \quad \abs{x}\to\infty.
\end{multline*}
This finishes the proof of Lemma \ref{5:asymp_lemma}.
\end{proof}

Last, we mention that $p$-Sobolev regular functions are H\"older.
\begin{lemma}\label{C:holder_cond}
For any $p\in (1/2, 3/2)$ and any $x, y\in\R$ we have that
\[
\abs{f(x)-f(y)}\lesssim \normH{p}{f}\abs{x-y}^{p-1/2}.
\]
\end{lemma}
\begin{proof}
The inequality follows directly from the estimate for the Fourier transform
\[
\abs{f(x)-f(y)} \le \int_\R\abs{\omega}^p\abs*{\hat{f}(\omega)}\abs*{\frac{e^{i\omega(x-y)}-1}{\abs{\omega}^p}}d\omega\le \normH{p}{f}\norm*{\frac{e^{i\omega(x-y)}-1}{\abs{\omega}^p}}_{L_2},
\]
where the second inequality is the Cauchy-Bunyakovsky-Schwarz inequality. The estimate by $\abs{x-y}^{p-1/2}$ follows from the change of variable. The condition $p\in (1/2, 3/2)$ ensures that the second factor on the right-hand side is finite.
\end{proof}

\subsection{Decomposition of $\mathcal{R}_f$ and trace norm estimates of the components}

Observe that it is sufficient to establish the inequality in Lemma \ref{2:diff_tr_est} for compactly supported functions. It may be extended then to all $H_2(\R)$ functions by continuity. Throughout this part we will follow this assumption.

The decomposition of the kernel $\mathcal{R}_f$ below represents it as a sum of three kernels, to which Proposition \ref{B:norm_est} is applicable. The proof is based on a direct substitution of the formula for $\mathfrak{D}$ instead of the hypergeometric functions in the kernel $\mathcal{R}_f$.

\begin{lemma}\label{5:ker_expr}
For a compactly supported $f\in H_2(\R)$ the operator $\I_{[1, +\infty)}\mathcal{R}_f\I_{[1, +\infty)}$ is an integral operator with the kernel equal to
\[
\tilde{\mathcal{R}}_f(x, y) = S_f(x, y) - T_f(x, y)-\overline{T_{\bar{f}}(y, x)}+Z_f(x, y),
\]
where
\begin{align*}
&S_f(x, y) = \int_{-\infty}^\infty \overline{\mathfrak{D}(yt)}\mathfrak{D}(xt)f_0(t)dt,\\
&T_f(x, y) = \frac{i}{\sqrt{2\pi}}\int_\R\overline{\mathfrak{D}(yt)}\frac{e^{-ixt}}{x}f'(t)dt + \frac{i}{\sqrt{2\pi}}\int_\R y\left[\overline{\mathfrak{D}'(yt)}-\abs{s}^2\frac{e^{iyt}}{yt}\right]\frac{e^{-ixt}}{x}f_0(t)dt,\\
&Z_f(x, y) = \frac{\abs{s}^2}{\sqrt{2\pi}}\int_\R\frac{e^{i(y-x)t}}{xy}\frac{d}{dt}\left(\frac{f_0(t)}{t}\right)dt.
\end{align*}
\end{lemma}
\begin{proof}
For a compactly supported bounded Borel $f$ we have that $\mathcal{R}_f$ is an integral operator with the kernel
\begin{equation}\label{5_eq:R_ker}
\mathcal{R}_f(x, y) = \int_\R\bra*{\CT(xt)\overline{\CT(yt)}-\frac{e^{i(y-x)t}}{\sqrt{2\pi}}}f(t)dt.
\end{equation}
It is therefore sufficient to establish that for any compact $A, B\subset [1, +\infty)$ we have
\[
\int_{A\times B}\bra*{\mathcal{R}_f(x, y)-\tilde{\mathcal{R}}_f(x, y)}dxdy=0.
\]

To show the equality we consider $\mathcal{R}_{f_0}$ instead of $\mathcal{R}_f$. Observe that these operators are equal, though we may not immediately write the kernel of the former as in the formula \eqref{5_eq:R_ker}. However, the fact that $\mathcal{R}_{f_0}$ is the weak limit of $\mathcal{R}_{f_0\I_{[-R, R]}}$ as $R\to\infty$ holds. This allows us to express
\begin{multline*}
\int_{A\times B}\mathcal{R}_f(x, y)dxdy = \langle \I_A, \mathcal{R}_f\I_b\rangle = \langle \I_A, \mathcal{R}_{f_0}\I_B\rangle = \lim_{R\to\infty}\langle \I_A, \mathcal{R}_{f_0\I_{[-R, R]}}\I_B\rangle = \\
= \lim_{R\to\infty}\int_{A\times B}\mathcal{R}_{f_0\I_{[-R, R]}}(x, y)dxdy.
\end{multline*}
We conclude that it is sufficient to show that
\begin{equation}\label{5_eq:ker_lem_1}
\lim_{R\to\infty}\int_{A\times B}\mathcal{R}_{f_0\I_{[-R, R]}}(x, y)dxdy = \int_{A\times B}\tilde{\mathcal{R}}_f(x, y)dxdy.
\end{equation}

The proof of equality consists of decomposition of the kernel $\mathcal{R}_-$ into three terms and their integration by parts. Express it as follows
\[
\mathcal{R}_{f_0\I_{[-R, R]}}(x, y) = S_{f_0\I_{[-R, R]}}(x, y) + L_{f_0\I_{[-R, R]}}(x, y) + \overline{L_{\bar{f}\I_{[-R, R]}}(y, x)},
\]
where
\[
L_f(x, y) = \int_\R\overline{\mathfrak{D}(yt)}\frac{e^{-ixt}}{\sqrt{2\pi}}f_0(t)dt.
\]
Integration by parts gives that
\[
L_f(x, y) = i\int_{-R}^R\overline{\mathfrak{D}(yt)}f_0(t)\frac{d}{dt}\frac{e^{-ixt}}{\sqrt{2\pi}x}dt 
=
-\tilde{T}_{f, R}-\frac{i\abs{s}^2}{\sqrt{2\pi}}\int_{-R}^R\frac{ye^{i(y-x)t}}{xyt}f_0(t)dt,
\]
where
\begin{multline*}
\tilde{T}_{f, R} = \frac{i}{\sqrt{2\pi}}\int_{-R}^R\overline{\mathfrak{D}(yt)}\frac{e^{-ixt}}{x}f'(t)dt + \frac{i}{\sqrt{2\pi}}\int_{-R}^R y\left[\overline{\mathfrak{D}'(yt)}-\abs{s}^2\frac{e^{iyt}}{yt}\right]\frac{e^{-ixt}}{x}f_0(t)dt-\\
-\frac{i}{\sqrt{2\pi}x}\bra*{\overline{\mathfrak{D}(yR)}f_0(R)e^{-ixR}-\overline{\mathfrak{D}(-yR)}f_0(-R)e^{ixR}}.
\end{multline*}
We note that though the function $\mathfrak{D}$ is not absolutely continuous, the function $\mathfrak{D}(x)f_0(x)$ is, which justifies the integration by parts.

Observe that the sum of the last kernel in the expression for $L_f$ with its conjugate gives
\[
-\frac{i\abs{s}^2}{\sqrt{2\pi}}\int_{-R}^R\frac{ye^{i(y-x)t}}{xyt}f_0(t)dt + \frac{i\abs{s}^2}{\sqrt{2\pi}}\int_{-R}^R\frac{xe^{i(y-x)t}}{xyt}f_0(t)dt = -\frac{\abs{s}^2}{\sqrt{2\pi}xy}\int_{-R}^R\frac{f_0(t)}{t}\frac{d}{dt}e^{i(y-x)t}dt=\tilde{Z}_{f, R},
\]
where by integration by parts the right-hand side is equal to
\[
\tilde{Z}_{f, R} = -\frac{\abs{s}^2}{Rxy\sqrt{2\pi}}\bra*{f_0(R)e^{i(y-x)R}+f_0(-R)e^{-i(y-x)R}} + \frac{\abs{s}^2}{xy\sqrt{2\pi}}\int_{-R}^Re^{i(y-x)t}\frac{d}{dt}\bra*{\frac{f_0(t)}{t}}dt.
\]

We conclude that
\[
\mathcal{R}_{f_0\I_{[-R, R]}}(x, y) = S_{f_0\I_{[-R, R]}}(x, y) - \tilde{T}_{f, R}(x, y) - \overline{\tilde{T}_{\bar{f}, R}(x, y)} + \tilde{Z}_{f, R}.
\]
The equality \eqref{5_eq:ker_lem_1} follows from the convergence
\begin{align*}
&\int_{A\times B}S_{f_0\I_{[-R, R]}}(x, y)dxdy \to \int_{A\times B}S_{f_0}(x, y)dxdy,\\
&\int_{A\times B}\tilde{T}_{f, R}(x, y)dxdy \to \int_{A\times B}T_f(x, y)dxdy,\\
&\int_{A\times B}\tilde{Z}_{f, R}(x, y)dxdy \to \int_{A\times B}Z_f(x, y)dxdy,
\end{align*}
as $R\to\infty$. Lemma \ref{5:asymp_lemma} and the assumption on the function $f$ yield locally uniform on $(x, y)\in (0, +\infty)^2$ convergence of these kernels. This finishes the proof of Lemma \ref{5:ker_expr}.
\end{proof}

Before diving into further calculations let us note that
\[
\normH{p_1}{f}\lesssim \normH{p_2}{f}+\normH{p_3}{f}
\]
for any $p_2\le p_1\le p_3$.

A direct application of Proposition \ref{B:norm_est} to the kernels in Lemma \ref{5:ker_expr} gives.

\begin{lemma}\label{5:tr_est}
We have
\begin{align*}
&\norm{\I_{[1, +\infty)}S_f\I_{[1, +\infty)}}_{\mathcal{J}_1}\lesssim \normH{1}{f},\\
&\norm{\I_{[1, +\infty)}T_f\I_{[1, +\infty)}}_{\mathcal{J}_1}\lesssim \normH{3/4}{f}+\normH{2}{f},\\
&\norm{\I_{[1, +\infty)}Z_f\I_{[1, +\infty)}}_{\mathcal{J}_1}\lesssim \normH{1}{f}.
\end{align*}
\end{lemma}
\begin{proof}
\textbf{Estimate for $S_f$}

By Lemma \ref{5:asymp_lemma} we have
\begin{equation}\label{5_eq:est_lem_1}
    \int_1^{+\infty}\abs{\mathfrak{D}(xt)}^2dx \lesssim \frac{1}{\abs{t}}\int_{\abs{t}}^{+\infty}\frac{1}{x(1+\sqrt{x})^2}dx \lesssim \frac{\abs{\ln\abs{t}}}{\abs{t}(1+\abs{t})}.
    \end{equation}
    Thereby using Proposition \ref{B:norm_est} and Lemma \ref{C:holder_cond} we conclude
    \[
 \norm{\mathbb{I}_{[1, \infty)}S_f\mathbb{I}_{[1, \infty)}}_{\mathcal{J}_1} \lesssim \int_\R\frac{\abs{\ln\abs{t}}}{\abs{t}(1+\abs{t})}f_0(t)dt \lesssim \normH{1}{f}\int_\R\frac{\abs{\ln\abs{t}}}{\abs{t}^{1/2}(1+\abs{t})}dt.
 \]
 
\textbf{Estimate for $T_f$}

The kernel $T_f$ consists of two terms. The estimate for the first term follows from Proposition~\ref{B:norm_est} and the inequality \eqref{5_eq:est_lem_1} directly
\[
\int_\R\sqrt{\int_1^\infty \abs{\mathfrak{D}(xt)}^2dx}\sqrt{\int_1^\infty\frac{1}{y^2}dy}\abs{f'(t)}dt \lesssim \int_\R \frac{\abs{f'(t)}\abs{\ln\abs{t}}^{1/2}}{t^{1/2}(1+\abs{t}^{1/2})}dt.
\]
The Cauchy-Bunyakovsky-Schwarz inequality yields
\[
\int_{\R\setminus [-1, 1]} \frac{\abs{f'(t)}\abs{\ln\abs{t}}^{1/2}}{t^{1/2}(1+\abs{t}^{1/2})}dt \lesssim \normH{1}{f}\sqrt{2\int_1^\infty \frac{\ln t}{t(1+t)}dt}.
\]
Last, by Lemma \ref{C:holder_cond} we conclude
\[
\int_{[-1, 1]} \frac{\abs{f'(t)}\abs{\ln\abs{t}}^{1/2}}{t^{1/2}(1+\abs{t}^{1/2})}dt \lesssim \normH{2}{f}.
\]

Let us proceed to the second term. From Lemma \ref{5:asymp_lemma} we have
\[
    \sqrt{\int_1^\infty \abs*{y\left(\mathfrak{D}'(yt)-\abs{s}^2\frac{e^{-iyt}}{yt}\right)}^2dx} \lesssim \sqrt{\frac{\abs{\ln\abs{t}}}{\abs{t}^3(1+\abs{t})}},
\]
from which we deduce by Lemma \ref{C:holder_cond}
    \begin{multline*}
    \int_\R \sqrt{\int_1^\infty x\abs*{\mathfrak{D}'(xt)-\abs{s}^2\frac{e^{-ixt}}{xt}}^2dx}\sqrt{\int_1^\infty\frac{1}{y^2}dy}\abs{f_0(t)}dt \lesssim \int_\R\abs{f_0(t)}\frac{\abs{\ln\abs{t}}^{1/2}}{\abs{t}^{3/2}\sqrt{(1+\abs{t})}}dt \lesssim\\
    \lesssim \normH{5/4}{f}\int_\R \frac{\abs{\ln\abs{t}}^{1/2}}{\abs{t}^{3/4}\sqrt{1+\abs{t}}}dt.
    \end{multline*}
    Proposition \ref{B:norm_est} finishes the proof for $T_f$.
    
\textbf{Estimate for $Z_f$}
    
    Express the derivative
    \[
    \frac{d}{dt}\left(\frac{f_0(t)}{t}\right) = \frac{f'(t)-f'(0)}{t}-\frac{1}{t^2}\int_0^tdu\int_0^uf''(v)dv.
    \]
    Lemma \ref{C:holder_cond} yields
    \begin{align*}
    &\abs*{\frac{d}{dt}\left(\frac{f_0(t)}{t}\right)}\lesssim \frac{\normH{2}{f}}{\sqrt{\abs{t}}},\\
    & \abs*{\frac{d}{dt}\left(\frac{f_0(t)}{t}\right)}\lesssim \abs*{\frac{f_0(t)}{t^2}}+\abs*{\frac{f'(t)}{t}}\lesssim \frac{\normH{1}{f}}{\abs{t}^{3/2}} + \frac{\abs{f'(t)}}{\abs{t}}.
    \end{align*}
    Using the first estimate for $t\in [-1, 1]$ and the second for $t\in \R\setminus [-1, 1]$ we conclude by Proposition~\ref{B:norm_est} that
    \[
    \norm{\I_{[1, \infty)}Z_f\I_{[1, \infty)}}_{\mathcal{J}_1} \lesssim \normH{1}{f}\int_{-1}^1 \frac{1}{\sqrt{\abs{t}}}dt + \normH{1}{f}\int_{\R\setminus [-1, 1]}\frac{1}{\abs{t}^{3/2}}dt + \int_{\R\setminus [-1, 1]}\frac{\abs{f'(t)}}{\abs{t}}dt\lesssim \normH{1}{f}.
    \]
    
    Lemma \ref{5:tr_est} is proved.
\end{proof}

\subsection{Trace class condition for $\I_{[0, 1]}\mathcal{R}_f$}

\begin{lemma}\label{5:01_est}
For $f\in L_\infty(\R)$ such that $x^2f(x)\in L_2(\R)$ we have that $\I_{[0, 1]}\CT f\in\mathcal{J}_1$.
\end{lemma}
\begin{proof}
Express the multiplication by $\mathbb{I}_{[0, 1]}$ as follows
\[
\mathbb{I}_{[0, 1]} = \mathbb{I}_{[0, 1]}x^{\bar{s}}x^{-\bar{s}}\mathbb{I}_{[0, 1]} = x^{\bar{s}}P_{\I_{[0, 1]}}x^{-\bar{s}} + \mathbb{I}_{[0, 1]}x^{\bar{s}}V\partial_xx^{-\bar{s}}\mathbb{I}_{[0, 1]},
\]
where $P_{\I_{[0, 1]}}$ is a one-dimensional orthogonal projector on $\I_{[0, 1]}$ and $V$ is the Volterra operator, defined by the formula
\[
(Vf)(x) = \int_0^xf(t)dt.
\]
A direct calculation shows that $x^{\bar{s}}V$ is Hilbert-Schmidt. The claim is thereby proved once one shows that $\abs{x}^{-\bar{s}}\mathbb{I}_{[0, 1]}\CT f$ is bounded and $\mathbb{I}_{[0, 1]}\partial_xx^{-\bar{s}}\CT f$ is Hilbert-Schmidt.

To show that $x^{-\bar{s}}\CT f$ is bounded observe that since $\rho(x)\psi(x) = x^{\bar{s}}e^{-\frac{i\pi}{2}\bar{s}}$ we have that its kernel is
\[
K_{\I_{[0, 1]}x^{-\bar{s}}\CT f}(x, y) = e^{-\frac{i\pi}{2}\bar{s}}\frac{e^{-ixy}}{\sqrt{2\pi}}Z_s(xy)f(y).
\]
By Lemma \ref{5:asymp_lemma} we have that
\[
\abs*{K_{\I_{[0, 1]}x^{-\bar{s}}\CT f}(x, y)}\lesssim\begin{cases} 
(1+\abs{xy}^{-\bar{s}})\abs{f(y)}, \Re s \le 0,\\
\abs{f(y)}, \Re s \ge 0.
\end{cases}
\]
We conclude that $K_{\I_{[0, 1]}x^{-\bar{s}}\CT f}(x, y)$ is square integrable on $[0, 1]\times \R$ under the assumptions of the lemma.

The kernel of the operator $N=\I_{[0, 1]}\partial_xx^{-\bar{s}}\CT f$ is equal to
\[
N(x, y) = -iK_{\I_{[0, 1]}x^{-\bar{s}}\CT f}(x, y)yf(y) + \frac{e^{-ixy}}{\sqrt{2\pi}}e^{-\frac{i\pi}{2}\bar{s}}Z_s'(xy)yf(y).
\]
The first kernel is Hilbert-Schmidt by the above argument. For the second kernel we use the estimate
\[
\abs{Z_s'(x)}\lesssim 1,
\]
which follows from the derivative formula \eqref{5_eq:5_1_der} and the asymptotic expansion \eqref{5_eq:D_asymp}. Therefore the second term is also Hilbert-Schmidt by a direct calculation. Lemma \ref{5:01_est} is proven completely.
\end{proof}

\subsection{Proof of Lemma \ref{2:diff_tr_est}}

\begin{proof}[Proof of Lemma \ref{2:diff_tr_est}]
The second claim is a direct consequence of Lemmata \ref{5:ker_expr}, \ref{5:tr_est}. We again note that, given the inequality is proven for compactly supported functions, it follows for all $H_2(\R)$ functions from extension by continuity. This also implies that $\I_{[1, \infty)}\mathcal{R}_f\I_{[1, \infty)}$ is trace class for any $f\in H_2(\R)$.

To prove the first claim express the operator $\mathcal{R}_f$ as follows
\[
\mathcal{R}_f = \I_{[1, +\infty)}\mathcal{R}_f\I_{[1, +\infty)} + \I_{[1, +\infty)}\mathcal{R}_f\I_{[0, 1]}+ \I_{[0, 1]}\mathcal{R}_f.
\]
The first summand is trace class by the second claim. The second and last term are trace class by Lemma \ref{5:01_est}.
\end{proof}

\section{Continuity of the remainder $Q(f)$}\label{sect:rem}
This section is devoted to the proof of Lemma \ref{2:remainder_continuity}. First introduce the notation and recall that
\begin{align*}
&Q(f) = Y(f)\exp\left(\Tr\mathbb{I}_{[1, \infty)}\mathcal{R}_f\mathbb{I}_{[1, \infty)}\right),\\
&Y(f) = \det(I+\mathcal{K}),\\
&\mathcal{K} = \I_{[1, +\infty)}(W_{e^{f_-}}\CO{e^{-f_+}}\CO{e^{-f_-}}W_{e^{f_+}}-I)\I_{[1, +\infty)}.
\end{align*}
\begin{lemma}
For $f\in \F^*L_1(\R)$ we have that
\[
\mathcal{K} = \I_{[1, +\infty)}W_{e^{f_-}}\I_{[1, +\infty)}\bra*{[\CO{e^{-f_+}},\CO{e^{-f_-}}]+\mathcal{R}_{e^{-f}}}\I_{[1, +\infty)}W_{e^{f_+}}\I_{[1, +\infty)}.
\]
\end{lemma}
\begin{proof}
Indeed, by the last claim of Theorem \ref{4:WH_properties} we have
\[
\I_{[1, +\infty)}W_{e^{f_-}} = \I_{[1, +\infty)}W_{e^{f_-}}\I_{[1, +\infty)}, \quad W_{e^{f_+}}\I_{[1, +\infty)} = \I_{[1, +\infty)}W_{e^{f_+}}\I_{[1, +\infty)}.
\]
Further, the first claim of Theorem \ref{4:WH_properties} also yields that
\[
G_{e^{-f}} = G_{e^{-f_-}}G_{e^{-f_+}}, \quad W_{e^{f_-}}W_{e^{-f}}W_{e^{f_+}}=I.
\]
Substituting directly the commutator
\[
G_{e^{-f_+}}G_{e^{-f_-}} = [G_{e^{-f_+}},G_{e^{-f_-}}]+G_{e^{-f_-}}G_{e^{-f_+}}
\]
into the expression for $\mathcal{K}$ and using the mentioned properties one obtains the assertion of the lemma.
\end{proof}

\begin{lemma}\label{6:comm_formula}
We have for arbitrary constants $\alpha, \beta$ that
\[
[\CO{e^{-f_+}},\CO{e^{-f_-}}] = \I_+\CT (e^{-f_+}+\alpha)\abs{x}^{2i\Im s}\F^*\I_-\F \abs{x}^{-2i\Im s}(e^{-f_-}+\beta)\CT^*\I_+.
\]
\end{lemma}
\begin{proof}
Recall the definition of $G_f$
\[
G_f = \I_+\CT f\CT^*\I_+.
\]
Observe that $G_{\alpha} = \alpha \I_+$ by Theorem \ref{3:Diagonalization}. This operator commutes with $G_f$ so we may assume that $\alpha, \beta$ are equal to zero.
Substituting the definition into the expression for the commutator and using the first claim of Theorem \ref{4:WH_properties} we get
\begin{multline*}
[\CO{e^{-f_+}},\CO{e^{-f_-}}] =\I_+\CT e^{-f_+}\CT^*\I_+\CT e^{-f_-}\CT^*\I_+ -\\
- \I_+ \CT e^{-f_+} \CT^*\CT e^{-f_-}\CT^*\I_+ = \I_+\CT e^{-f_+} \CT^*\I_- \CT e^{-f_-}\CT^*\I_+.
\end{multline*}
Theorem \ref{3:PW_th} finishes the proof.
\end{proof}

Recall the notation $f_0(x) = f(x)-f(0)$.

\begin{lemma}\label{6:rem_est}
We have that
\[
\norm{\I_{[1, +\infty)} \CT f_0\abs{x}^{2i\Im s}\F^*\I_-}_{\mathcal{J}_2}\lesssim \normH{3/4}{f}+\normH{2}{f}.
\]
\end{lemma}
\begin{proof}
We first show that
\[
\norm{\I_{[1, +\infty)} (\CT -\F e^{i\pi \Re s\sign x})f_0}_{\mathcal{J}_2}\lesssim \normH{3/4}{f}.
\]
Recall the definition
\[
\CT(x) = \frac{e^{-ix}}{\sqrt{2\pi}}\psi(x)\rho(x)Z_s(x)e^{i\pi\Re s \sign x}, \quad \rho(x)\psi(x)=x^{\bar{s}}e^{-\frac{i\pi}{2}\bar{s}}.
\]
By Lemma \ref{5:asymp_lemma} we have that
\[
\abs*{\CT(x) - \frac{e^{-ix}}{\sqrt{2\pi}}e^{i\pi\Re s\sign x}} \lesssim \frac{\abs{x}^{\Re s}}{1+\abs{x}^{1+\Re s}}.
\]
Let $K$ stand for the kernel of the operator $\I_{[1, +\infty)} (\CT -\F e^{i\pi \Re s\sign x})f_0$. The above inequality gives
\[
\abs{K(x, y)}\lesssim \frac{\abs{xy}^{\Re s}\abs{f_0(y)}}{1+\abs{xy}^{1+\Re s}}.
\]

We now estimate the $L_2$ norm of the kernel. We immediately have that
\[
\int_1^\infty \frac{\abs{xy}^{2\Re s}}{(1+\abs{xy}^{1+\Re s})^2}dx \lesssim \frac{1}{\abs{y}(1+\abs{y})}.
\]
Substituting it into the integral over $y$ and using Lemma \ref{C:holder_cond} we conclude
\[
\int_1^\infty dx\int_{\R} dy \abs{K(x, y)}^2 \lesssim \int_{\R} \frac{\abs{f_0(y)}^2}{\abs{y}(1+\abs{y})}dy\lesssim \normH{3/4}{f}^2.
\]

Next we show that
\[
\norm{\I_{[1, +\infty)}\F f_0e^{i\pi\Re s \sign x}\abs{x}^{-2i\Im s}\F^*\I_-}_{\mathcal{J}_2}\lesssim \normH{3/4}{f}+\normH{2}{f}.
\]
A direct calculation gives that for any $1$-Sobolev regular function $f$ we have
\[
\norm{\I_{[1, +\infty)}\F f\F^*\I_-}_{\mathcal{J}_2} \lesssim \normH{1}{f}.
\]
Observe that since $f_0(0)=0$ the function $f_0(x)e^{i\pi\Re s \sign x}\abs{x}^{2i\Im s}$ has square integrable derivative, which is at most
\[
\abs*{\bra*{f_0(x)e^{i\pi\Re s \sign x}\abs{x}^{2i\Im s}}'}\lesssim \abs{f'(x)} + \abs*{f_0(x)/x}.
\]
Thereby its $L_2$ norm may be estimated by
\[
\normH{1}{f_0(x)e^{i\pi\Re s \sign x}\abs{x}^{2i\Im s}} \lesssim \normH{1}{f} + \norm{f_0(x)/x}_{L_2},
\]
since the Parseval Theorem yields that 1-Sobolev seminorm is a constant times $L_2$ norm of the derivative. Let us estimate the second term. We have by Lemma \ref{C:holder_cond} that
\[
\int_{\R\setminus [-1, 1]}\abs*{\frac{f_0(x)}{x}}^2dx \lesssim \normH{3/4}{f}.
\]
Observe that for $t\in [-1, 1]$ we have
\[
\abs{f_0(t)} \le \norm{f'}_{L_\infty [-1, 1]}\abs{t},
\]
where
\[
\norm{f'}_{L_\infty [-1, 1]} \le \norm{\lambda\hat{f}(\lambda)}_{L_1} \lesssim \normH{1}{f} + \normH{2}{f},
\]
by the Cauchy-Bunyakovsky-Schwarz inequality. This finishes the proof of Lemma \ref{6:rem_est}.
\end{proof}

Before proving the main result of the section let us state several properties of the Sobolev spaces $H_p(\R)$, which we prove in the end of the section.

\begin{proposition}\label{6:sob_prop}
\begin{itemize}
\item We have that $H_p(\R)\subset H_q(\R)$ if $p<q$.
\item We have that $H_p(\R)$ is a Banach algebra for $p>1/2$ if endowed with a norm, equal to a constant times $\norm{\cdot}_{H_p}$.
\item Consequently, $e^f-1\in H_p(\R)$ if $f\in H_p(\R)$, $p>1/2$.
\item We have that for any $p\in (0, 1)$
\[
\frac{1}{2}\int_{\R^2}\frac{\abs{f(x)-f(y)}^2}{\abs{x-y}^{1+2p}}dxdy = C_p\normH{p}{f}^2,
\]
where
\[
C_p = 2\pi\int_\R\frac{1-\cos x}{\abs{x}^{1+2p}}dx.
\]
\item We have that for $p\in (0, 1)$
\[
\normH{p}{e^f}\le e^{2\norm{f}_{L_\infty}}\normH{p}{f}.
\]
\item The following inequalities hold
\begin{align*}
&\normH{1}{e^f}\le e^{\norm{f}_{L_\infty}}\normH{1}{f},\\
&\normH{2}{e^f}\le e^{\norm{f}_{L_\infty}}(\normH{2}{f}+\normH{1}{f}(\normH{1}{f}+\normH{2}{f})).
\end{align*}
\end{itemize}
\end{proposition}

We are ready to prove the main result of the section.
\begin{proof}[Proof of Lemma \ref{2:remainder_continuity}]
Since $H_p(\R)$ forms a Banach algebra, we have that $\norm{e^f-1}_{H_p}$ is continuous with respect to the norm $\norm{f}_{H_p}$, as well as the norm of $\norm{f_\pm}_{H_p}$ is continuous with respect to the norm $\norm{f}_{H_p}$. Thereby by Lemma \ref{2:diff_tr_est} the operator $\I_{[1, +\infty)}\mathcal{R}_{e^{-f}}\I_{[1, +\infty)}$ is $\mathcal{J}_1$-continuous with respect to $\norm{f}_{H_2}$. Further, the Wiener-Hopf operators $W_{e^{f_\pm}}$ are continuous with respect to the norm $\norm{\hat{f}_\pm}_{L_1}$, which is continuous with respect to $\norm{f}_{H_2}$. 
Recall the inequalities
\begin{equation}\label{6_eq:norm_ineq}
\begin{aligned}
&\norm{K_1K_2}_{\mathcal{J}_1}\le \norm{K_1}_{\mathcal{J}_1}\norm{K_2}_{\mathcal{J}_2}, \quad K_1, K_2\in\mathcal{J}_2, \\
&\norm{K_3B}_{\mathcal{J}_1}\le \norm{K_3}_{\mathcal{J}_1}\norm{B}, \quad K_3\in\mathcal{J}_1, B \text{ is bounded.}
\end{aligned}
\end{equation}
We conclude that the operator
\[
\I_{[1, +\infty)}W_{e^{f_-}}\I_{[1, +\infty)}\mathcal{R}_{e^{-f}}\I_{[1, +\infty)}W_{e^{f_+}}\I_{[1, +\infty)}
\]
is $\mathcal{J}_1$-continuous with respect to $\norm{f}_{H_2}$.
Using the first inequality \eqref{6_eq:norm_ineq} and Lemmata \ref{6:rem_est}, \ref{6:comm_formula} we have that the operator
\[
\I_{[1, +\infty)}W_{e^{f_-}}\I_{[1, +\infty)}[G_{e^{-f_+}}, G_{e^{-f_-}}]\I_{[1, +\infty)}W_{e^{f_+}}\I_{[1, +\infty)}
\]
is $\mathcal{J}_1$-continuous with respect to $\norm{f}_{H_2}$. Thereby $\mathcal{K}$ is $\mathcal{J}_1$-continuous with respect to $\norm{f}_{H_2}$ and $Y(f)$ is continuous with respect to $\norm{f}_{H_2}$. The continuity of the exponential factor in the formula for $Q(f)$ follows from the second claim of Lemma \ref{2:diff_tr_est}. This proves the continuity of $Q(f)$ --- the first assertion of Lemma \ref{2:remainder_continuity}.

The second claim follows from the inequality
\[
\abs{\det(I+\mathcal{K})-1}\le \norm{\mathcal{K}}_{\mathcal{J}_1}e^{\norm{\mathcal{K}}_{\mathcal{J}_1}},
\]
The estimate for $\norm{\mathcal{K}}_{\mathcal{J}_1}$ is obtained from Lemma \ref{2:diff_tr_est} and the second inequality \eqref{6_eq:norm_ineq}, applied to the operator
\[
\I_{[1, +\infty)}W_{e^{f_-}}\I_{[1, +\infty)}\mathcal{R}_{e^{-f}}\I_{[1, +\infty)}W_{e^{f_+}}\I_{[1, +\infty)};
\]
and the first inequality \eqref{6_eq:norm_ineq}, Lemmata \ref{6:comm_formula}, \ref{6:rem_est} applied to the operator
\[
\I_{[1, +\infty)}W_{e^{f_-}}\I_{[1, +\infty)}[G_{e^{-f_+}}, G_{e^{-f_-}}]\I_{[1, +\infty)}W_{e^{f_+}}\I_{[1, +\infty)}.
\]
Last, one employs the last two statements of Proposition \ref{6:sob_prop}
\end{proof}

\begin{proof}[Proof of Proposition \ref{6:sob_prop}]
The first claim follows from a direct calculation. 

To prove the second claim first observe that
\[
\abs{\omega + \nu}^{2p}\le 2^{2p-1}(\abs{\omega}^{2p}+\abs{\nu}^{2p}).
\]
Thereby we may write for the $p$-seminorm
\[
	2^{1-p}\normH{p}{fg}^2 \le \int_{\R^2}d\omega d\nu\abs{\nu-\omega}^{2p}\abs{\hat{f}(\nu -\omega)\hat{g}(\omega)}^2d\omega d\nu + \int_{\R^2}d\omega d\nu\abs{\omega}^{2p}\abs{\hat{f}(\nu -\omega)\hat{g}(\omega)}^2d\omega,
\]
where the right-hand side is at most
\[
\normH{p}{f}^2\norm{\hat{g}}_{L_1} + \normH{p}{f}^2\norm{\hat{f}}_{L_1}
\]
by the Young convolution inequality. We are then left to observe that
\[
\norm{\hat{f}}_{L_1} \le \norm{f}_{L_2} + \frac{2}{\sqrt{2p-1}}\normH{p}{f},
\]
which follows from the application of Cauchy-Bunyakovsky-Schwarz inequality to two terms of
\[
\int_\R\hat{f}(\omega)d\omega \le \int_{[-1, 1]}\hat{f}(\omega)d\omega + \int_{\R\setminus[-1, 1]}\frac{1}{\abs{\omega}^p}\abs{\omega}^p\hat{f}(\omega)d\omega.
\]

Let us prove the fourth statement. After one expresses the left-hand side
\[
\frac{1}{2}\int_{\R^2}\frac{\abs{f(x)-f(y)}^2}{\abs{x-y}^{1+2p}}dxdy = \frac{1}{2}\int_\R \left(\int_\R \abs*{\frac{f(x+y)-f(x)}{\abs{y}^{1/2+p}}}^2dx\right)dy,
\]
uses the Parseval identity and observes that
\[
\frac{1}{2\pi}\int_\R e^{-i\omega x}\frac{f(x+y)-f(x)}{\abs{y}^{1/2+p}}dx = \frac{e^{i\omega y}-1}{\abs{y}^{1/2+p}}\hat{f}(\omega),
\]
one concludes
\[
\frac{1}{2}\int_{\R^2}\frac{\abs{f(x)-f(y)}^2}{\abs{x-y}^{1+2p}}dxdy = \pi \int_{\R^2}\frac{\abs{e^{i\omega y}-1}^2}{\abs{y}^{1+2p}}\abs{\hat{f}(\omega)}^2d\omega dy = C_p\normH{p}{f}^2.
\]

The fifth claim follows from the fourth one and inequality
\[
\abs{e^x-e^y}\le e^{x+y}\abs{x-y}.
\]

The last claim may be verified by calculating $L_2$ norms of derivatives.
\end{proof}

\section{Proof of Corollary \ref{1:KS_conv}}\label{sect:corr_proof}
Let us first recall the Feller smoothing estimate.
\begin{theorem}[{\cite[p. 538]{F_66}}]\label{7:feller_est}
Assume we are given two real-valued random variables with distribution functions $F_1$, $F_2$ and characteristic functions $\varphi_1$, $\varphi_2$. For any $T>0$ we have that
\[
\sup_{x\in \R}\abs{F_1(x)-F_2(x)} \le \frac{24}{\sqrt{2\pi^2}T} +\frac{1}{\pi}\int_{-T}^T\abs*{\frac{\varphi_1(y)-\varphi_2(y)}{y}}dy.
\]
\end{theorem}

Next recall the notation. By $\overline{S}_f$ we denoted the regularized additive functionals (see Definition~\ref{3:reg_def}). We defined
\[
    F_R(x) = \P^s(\overline{S}_{f(x/R)}\le x),\quad F_{\mathcal{N}}(x)=\frac{1}{\sqrt{2\pi}}\int_{-\infty}^xe^{-\frac{t^2}{2}}dt.
\]
Below using Theorem \ref{7:feller_est} and Theorem \ref{1:mult_formula} we give an estimate for the difference $\abs{F_R(x)-F_{\mathcal{N}}(x)}$.

\begin{proof}[Proof of Corollary \ref{1:KS_conv}]
Observe that the expression
\[
\int_0^\infty \omega\hat{f}(\omega)\hat{f}(-\omega)d\omega
\]
is invariant under the dilation of the argument. Further, we have
\[
\normH{p}{f(\cdot /R)} = R^{\frac{1}{2}-p}\normH{p}{f}.
\]
Denote $\varphi_R(k) = \E e^{ik\overline{S}_{f(\cdot /R)}}$, $\varphi_{\mathcal{N}}(k) = e^{-k^2/2}$. By Theorem \ref{1:mult_formula} we have
\[
\abs*{\frac{\varphi_R(k)-\varphi_{\mathcal{N}}(k)}{k}} = \abs*{\frac{Q(ikf(\cdot /R))-1}{k}} \le Ce^{CL(ikf(\cdot /R))}\frac{L(ikf(\cdot/R))}{\abs{k}},
\]
where for any $R \ge 2$ there exists a constant $\tilde{C}$ such that
\[
\frac{L(ikf(\cdot/R))}{\abs{k}} \le \frac{\tilde{C}}{R^{1/4}}(1+\abs{k})e^{\tilde{C}\abs{k}}, \quad Ce^{CL(ikf(\cdot /R))} \le \tilde{C}e^{R^{-1/4}\tilde{C}\abs{k}^2e^{\tilde{C}\abs{k}}}.
\]
We deduce that for some constant $C'$ the following inequality holds
\[
\int_{-T}^T\abs*{\frac{\varphi_R(k)-\varphi_{\mathcal{N}}(k)}{k}}dk \le \frac{2T^2C'}{R^{1/4}}e^{R^{-1/4}C'T^2e^{\tilde{C}T}}.
\]
The claim follows from substituting $T=\frac{1}{8}\ln R$ and using Theorem \ref{7:feller_est}.
\end{proof}

    \appendix
    \section{Trace class and Hilbert-Schmidt operators}\label{appendix:tr_class_op}

In this section we recall basic definitions and statements about ideals of trace class and Hilbert-Schmidt operators. We refer the reader \cite{S_05, S_15} for the detailed exposition of the theory.

Recall that for a compact operator $K\in\mathcal{K}(\mathcal{H})$ on a separable Hilbert space $\mathcal{H}$ its singular values $\{s_n(K)\}_{n \in \N}$ are defined as eigenvalues of $\abs{K}=\sqrt{K^*K}$. For $p>0$ introduce the space
\[
\mathcal{J}_p = \{K\in \mathcal{K}(\mathcal{H}): \sum_{n\in\N}\abs{s_n(K)}^p<+\infty\},
\]
which is a Banach space if endowed with the norm
\[
\norm{K}_{\mathcal{J}_p}^p = \sum_{n\in\N}\abs{s_n(K)}^p.
\]
These spaces for $p=1$ and $p=2$ are called spaces of trace class and Hilbert-Schmidt operators respectively.

For an operator $K\in\mathcal{J}_1(\mathcal{H})$ and an arbitrary orthonormal basis $\{e_j\}_{j\in \N}$ define its trace
\[
\Tr{K} = \sum_{j\in\N}\langle e_j, K e_j\rangle_{\mathcal{H}},
\]
which is a $\norm{\cdot}_{\mathcal{J}_1}$-continuous function and is independent of the choice of a basis. The Fredholm determinant is then defined by the formula
\[
\det(I+K) = \sum_{j\ge 0}\Tr (\wedge^k K).
\]
The determinant is again $\norm{\cdot}_{\mathcal{J}_1}$-continuous. In particular, we have
\begin{equation*}
	\abs{\det(I+K)-1}\le \norm{K}_{\mathcal{J}_1}e^{\norm{K}_{\mathcal{J}_1}}, \quad \abs{\Tr K}\le \norm{K}_{\mathcal{J}_1}.
\end{equation*}
As for the usual determinant, we have
\begin{equation*}
	\det e^K = e^{\Tr K}.
\end{equation*}
Further, the Fredholm determinant is invariant under conjugation by an invertible operator $X$:
\begin{equation*}
\det(X^{-1}(I+K)X) = \det(I+K).
\end{equation*}
For $K_1$, $K_2\in\mathcal{J}_1$ we have
\begin{equation*}
\det((I+K_1)(I+K_2))=\det(I+K_1)\det(I+K_2).
\end{equation*}

If an operator $K$ is Hilbert-Schmidt but not trace class, one may introduce a regularization of the Fredholm determinant. Approximate $K$ by its finite-dimensional projections $K_n\overset{\mathcal{J}_2}{\to} K$. Define
\[
\det{}_2(I+K)=\lim_{n\to\infty}\det(I+K_n)e^{-\Tr(K)},
\]
where the limit exists since the function on the right-hand side is $\norm{\cdot}_{\mathcal{J}_2}$-continuous. Their limit --- the function on the left-hand side --- is thereby also $\norm{\cdot}_{\mathcal{J}_2}$-continuous.

Last, we mention that the Hilbert-Schmidt norm of an integral operator $K$ coincides with the $L_2$-norm of its kernel
\[
\norm{K}_{\mathcal{J}_2}^2 = \int_{X^2}\abs{K(x, y)}^2d\mu(x)d\mu(y).
\]

\end{document}